\documentclass[12pt]{article}
\RequirePackage{hyperref}
\usepackage{amssymb, amsmath, amsthm,epsfig}
\usepackage{fullpage}
\usepackage{bm, bbm, mathrsfs}
\usepackage{color}
 \definecolor{purple}{rgb}{.8,0,.8}

\def\Section#1#2{\section{#2.}\label{#1}\setcounter{equation}{0}}
\def\sect#1{Section~\ref{#1}}

\newtheorem{theorem}{Theorem}

\newtheorem{proposition}[theorem]{Proposition}

\newtheorem{definition}[theorem]{Definition}

\newtheorem{example}[theorem]{Example}

\newcommand{\pf}{\begin{proof}}
\newcommand{\foorp}{\end{proof}}

\def\Rmk#1\par{\medskip
\emph{Remark}: \ #1 \par \medskip}

\newcommand{\pd}[2]{\frac{\partial#1}{\partial #2}}

\newcommand{\beq}{\begin{equation}}
\newcommand{\eeq}{\end{equation}}

\def\eq#1{{\rm (\ref{#1})}}
\def\eqf{formula~\eq}

\def\ostrut#1#2{\hbox{\vrule height #1pt depth #2pt width 0pt}}

\chardef\itxt="10             
\chardef\otxt="1C             
\chardef\Otxt="1F             
\def\vtxt#1{{\accent"14 #1}}  

\def\hexnumber#1{\ifcase#1 0\or1\or2\or3\or4\or5\or6\or7\or8\or9\or
 A\or B\or C\or D\or E\or F\fi}
\edef\msbhx{\hexnumber\symAMSb}   
\mathchardef\emptyset="0\msbhx3F

\newif\ifnames \namesfalse
\def\namez#1;{\namesfalse\namezz#1,;} 
\def\namezz#1,#2,#3;{\ifx,#3,\ifnames \ and\fi\fi #1,#2,\ifx,#3,\else 
    \namestrue\namezz #3;\fi}

\def\key#1 #2\par{\bibitem{#1} #2\par}

\def\book#1;#2;#3\par{\namez#1;{\it #2},#3.}
\def\paper#1;#2;#3; #4(#5)#6\par{\namez#1;#2,{\it #3}\if a#4, to appear.\else
     { \bf #4}(#5),#6.\fi}
\def\inbook#1;#2;#3;#4\par{\namez#1;#2, {\sl in\/}:{\it#3},#4.}
\def\appx#1;#2;#3;#4\par{\namez#1;#2, {\sl appendix in\/}:{\it #3},#4.\par }
\def\thesis#1;#2;#3\par{\namez#1;{\it #2}, Ph.D.~Thesis,#3.}
\def\uthesis#1;#2;#3\par{\namez#1;{\it #2}, Ph.D.~Thesis, University of Minnesota,#3.}
\def\preprint#1;#2;#3\par{\namez#1;#2, preprint,#3.}
\def\upreprint#1;#2;#3\par{\namez#1;#2, preprint, University of Minnesota,#3.}
\def\prepare#1;#2;#3\par{\namez#1;#2, in preparation.}
\def\personal#1;#2\par{\namez#1; personal communication,#2.}
\def\submit#1;#2;#3\par{\namez#1;#2, submitted.}
\def\other#1\par{#1.}

\def\nosemiz#1;{#1}

\def\rf#1{\rfzzz#1;\rfzzz}
\def\rfzzz#1;#2\rfzzz{\ifx;#2;\cite{#1}\else\cite[\nosemiz\ignorespaces#2]{#1}\fi}

\def\eg{e.g.,~}

\def\Eq#1$$#2$${\StEq#1  \EnEq{#2}}
\def\StEq#1 #2\EnEq#3{\begin{equation}\label{#1} #3\end{equation}}

\newdimen\eqjot \eqjot = 1\jot
\def\openupeq{\openup \the\eqjot}
\def\addtab#1={#1\;&=}

\def\caeq#1#2{{\def\\{\cr}\vcenter{\openupeq \halign
    {$\hfil\displaystyle ##\hfil$&&$\hfil\hskip#1pt\displaystyle ##\hfil$\cr#2\cr}}}}
\def\ceq{\caeq{20}}  

\def\iaeq#1#2#3{{\def\\{\cr\hskip#1pt}\vcenter{\openupeq\halign{$\displaystyle
   ##\hfil$&&\hskip#2pt$\displaystyle##\hfil$\cr #3\cr}}}}
\def\ibeq#1{\iaeq{#1}{#1}}

\def\saeq#1#2{{\def\\{\cr}\vcenter{\openup1\jot \halign{$\displaystyle
   ##\hfil$&&\hskip#1pt$\displaystyle##\hfil$\cr #2\cr}}}}
 
\def\sxeq{\saeq{30}} 
\def\seq{\saeq{20}}  

\def\ezeq#1#2#3{{\def\\{\cr#1}\vcenter{\openupeq \halign{$\displaystyle 
   \hfil##$&$\displaystyle##\hfil$&&\hskip#2pt$\displaystyle##\hfil$\cr#1#3\cr}}}}
\def\eaeq{\ezeq\addtab}

\def\eeqo{\eaeq{20}}  

\def\qaeq#1#2{{\def\\{&}\vcenter{\openupeq\halign{$\displaystyle
   ##\hfil$&&\hskip#1pt$\displaystyle##\hfil$\cr #2\cr}}}}
\def\qqeq{\qaeq{40}} 
\def\qxeq{\qaeq{30}} 
\def\qeq{\qaeq{20}} \def\weq{\qaeq{15}}
\def\xeq{\qaeq{10}}

\def\ro#1{{\rm #1}} 

\def\:{\mskip2mu}

\def\xbox#1{\ \mbox{#1} \ }
\def\qbox#1{\quad \mbox{#1} \quad}
\def\qqbox#1{\qquad \mbox{#1} \qquad}

\def\f#1{{\textstyle \frac1{#1}}} \def\fr#1#2{{\textstyle \frac{#1}{#2}}}

\def\ophantom#1#2{\setbox0=\hbox{$#1#2$}\setbox2 = \null
                 \ht2 = \ht0 \dp2 = \dp0 \box2}
\def\odphantom#1{\ophantom\displaystyle{#1}}

\def\set#1#2{\left \{\,#1\odphantom{#2}\;\right | \left .\;#2\odphantom{#1}\,\right \}}

\def\comp{\raise 1pt \hbox{$\,\scriptstyle\circ\,$}}

\mathchardef\oplussy="2208
\def\oplus{\raise 1pt \hbox{$\,\scriptscriptstyle\oplussy\,$}}

\def\longmapstox{\ \longmapsto\ } 

\def\barr#1{\overline{#1}{}}

      \let\Pa\Parens
  \let\bc\braces

             \let\bbk\bbrackets

 \let\abs\norm

\def\Summ#1#2#3{\sum _{#1\:=\:#2}^{#3}\>}
\def\Sum#1#2{\Summ{#1}1{#2}}  

\def\pd#1#2{\mathchoice{\frac{\partial #1}{\partial #2}}{\partial #1/\partial #2}{\partial #1/\partial #2}{\partial #1/\partial #2}}
\def\pdo#1{\mathchoice{\frac{\partial}{\partial #1}}{\partial /\partial #1}{\partial /\partial #1}{\partial /\partial #1}}

\def\rg#1#2{#1=1,\ldots,#2} 
\def\subs #1#2{#1_1,\ldots,#1_{#2}}

\def\sups #1#2{#1^1,\ldots,#1^{#2}}
\def\psups #1#2{(#1^1,\ldots,#1^{#2})}

\def\at#1{|_{#1}}  

\newcommand{\R}{\mathbb{R}}

\def\GL#1{\ro{GL}(#1)}

\def\SL#1{\ro{SL}(#1)}

\def\PGL#1{\ro{PGL}(#1)}

\def\AG#1{\ro{A}(#1)}  
 
\def\SA#1{\ro{SA}(#1)}

\def\Diff{\CD\hbox{\it i\char'13}\,}

\def\Upstar{^{\displaystyle *}}

\def\ps#1{^{(#1)}}
\def\pss#1{^{(#1)\,\displaystyle *}}

\def\M{M}    \def\m{m}
\def\N{N} \def\n{n} 

\def\Ih{\widehat I} \def\Fh{\widehat F}

\def\giso#1{G\Upstar _{#1}}
\def\GN{\giso \N} \def\Gt{\widetilde G}

\def\Gw{G_\w} \def\GS{G_S}
\def\Cp{C} \def\C{\widehat C}
\def\vb{{\mathbf v}} 
\def\b{{\mathbf b}}   
\def\c{{\mathbf c}} \def\ch{{\widehat \c}}

\def\P{\Pi} \def\Pk{\P\ps k}    
\def\Pks{\P\pss k}    

\def\Po{{\displaystyle \P^{\,}_0}} \def\Pok{{\displaystyle \P\ps k_0}}
\def\Pc{{\displaystyle \P^{\,}_\ch}} \def\Pck{{\displaystyle \P\ps k_\ch}}
\def\Pb{{\displaystyle \P^{\,}_\b}} 
\def\Tc{{\displaystyle T^{\,}_\c}}

\def\F{\mathcal F}
\def\GP{{\cal G}_{\P}}

\def\Cc#1{{\rm C}^{#1}} \def\Cci{\Cc \infty }

\def\J#1{{\rm J}^{#1}}  \def\Jk{\J k} \def\Jl{\J l} \def\Ji{\J\infty }
\def\j#1{{\rm j}_{#1}}  \def\jk{\j k}  

\def\Jt#1{{\displaystyle {\rm J}_{\P}^{#1}}} 
  \def\Jtk{\Jt k} 
\def\Jto#1{{\displaystyle {\rm J}_{\Po}^{#1}}} 
  \def\Jtok{\Jto k} 

\def\piH{\pi_H} \def\dH{d_H} \def\dV{d_V}

\def\inv{\iota} 
\def\invh{\widehat \inv\,}
\def\invt{\widetilde \inv\,}

\def\K{\mathcal{K}} 
\def\Kt{\widetilde{\K}}
\def\Kh{\widehat{\K}}

\def\U{\mathcal{U}} \def\Uh{\widehat \U}
\def\cF{\Cci} 
\def\cFU{\cF(\U)} \def\cFUG{\cFU^G}
\def\cFUh{\cF(\Uh)} \def\cFUhG{\cFUh^G}

\def\CD{\mathcal{D}} 
\def\CDh{\widehat{\CD}}

\def\CE{\CDh}

\def\gk{g\ps k} \def\Gk{G\ps k} 

\def\z{z} \def\zk{\z\ps k} 
 
\def\zt{\widetilde \z} 
\def\w{v} \def\wk{\w\ps k}  
\def\v{u}   
\def\y{w} \def\yk{\y\ps k}  
\def\rr{w}

\def\lomega{\widehat\omega }
\def\dom{\ro{dom}\,} \def\ker{\ro{ker}\,} \def\rank{\ro{rank}\,}
\def\Ad{\ro{Ad}\,}  

\def\bp#1#2#3{[\:#1,#2,#3\:]}
\def\tp#1#2#3{#1\cdot #2 \times #3}

\def\kh{\widehat \kappa} \def\th{\widehat \tau} \def\bh{\widehat \beta}

\def\prcurv{\eta} \def\prcurvh{\widehat\eta} \def\pral{\xi}
\def\eacurv{\mu}  \def\eaal{\chi}
\def\affcurv{\nu} \def\affcurvh{\widehat\nu} \def\affal{\varrho}
\def\ihZ{\widehat \zeta}
\def\itZ{\widetilde \zeta}
\def\caal{\sigma} 
\def\pbal{{\widehat\xi}}
\def\A{A}  \def\B{B}

\def\ah{\widehat \alpha} \def\bh{\widehat \beta}

\def\orbit{\mathcal O}

\newcommand{\RP}{\mathbb{RP}}

\def\Fhat{\widehat F}

\begin{document}


\title{Invariants of objects and their images \break under surjective maps}

\maketitle

\begin{center}
\begin{tabular*}{.95\textwidth}{@{\extracolsep{\fill}} ll}
Irina A.~Kogan\footnotemark[1] & Peter J. Olver\footnotemark[2]  \\
Department of Mathematics & School of Mathematics \\
North Carolina State University&University of Minnesota \\
Raleigh, NC \quad 27695&Minneapolis, MN \quad 55455 \\
{\tt iakogan@ncsu.edu} & {\tt olver@umn.edu} \\
\url{http://www.math.ncsu.edu/~iakogan} &\url{http://www.math.umn.edu/~olver}
\end{tabular*}
\end{center}

\footnotetext[1]{{\it Supported in part by NSF Grants CCF 13--19632 and DMS 13--11743.}}
\footnotetext[2]{{\it Supported in part by NSF Grant DMS 11--08894.}}

\baselineskip=16pt
\vskip5mm
\noindent{\bf Abstract:} We examine the relationships between  the differential invariants  of objects and of their images  under a surjective map. We analyze  both the case when the underlying transformation group is projectable and hence induces an action on the image,  and the case when only a proper subgroup of the entire group acts projectably. In the former case, we  establish a  constructible isomorphism between the algebra of differential invariants of the images and the algebra of fiber-wise constant (gauge) differential invariants of the objects. In the latter case, we describe residual effects of the   full transformation group on the image invariants.  Our motivation comes from the problem of reconstruction of an object from multiple-view images, with  central and parallel projections of curves  from three-dimensional space to the two-dimensional plane serving as our main examples. 

\vskip.3in

\Section1{Introduction}

The subject of this paper is the behavior of invariants and, particularly, differential invariants under surjective maps.  While our theoretical results are valid for manifolds of arbitrary dimension, the motivating examples are central and parallel projections from three-dimensional space onto the two-dimensional plane, as prescribed by simple cameras. 
We concentrate on the effect of such projections on space curves, leaving the analysis of surfaces to subsequent investigations.  We will, in particular, derive relatively simple formulas relating the centro-affine invariants of a space curve, as classified in \cite{Ocentro}, to the projective curvature invariant of its projections.

The relationship between three-dimensional objects and their two-dimensional images under projection is a problem of major importance in image processing, and covers a broad spectrum of fundamental issues in computer vision, including stereo vision, structure from motion, shape from shading, projective invariants, etc.; see, for example, \cite{ArStSt,Bertalmio,CipGib, Faugeras,HannHick,MSKS,Sapiro}.  Our focus on differential invariants is motivated by the method of differential invariant signatures, \cite{COSTH}, used to classify objects up to group transformations, including rigid motions, and equi-affine, affine, centro-affine, and projective maps.  Our analysis is founded on the method of equivariant moving frames, as first proposed in \cite{FOmcII}, and we will assume that the reader is familiar with the basic techniques.  See \cite{Mansfield, Omfl} for recent surveys of the method and many of its applications.  In \cite{HubertKogan-focm, HubertKogan-jsc}, an algebraic interpretation of the equivariant moving frame was developed, leading to an algorithm for constructing a generating  set of rational invariants along with a set of algebraic invariants, that exhibit the replacement property.

A key problem in mathematics, arising, for example, in geometry, invariant theory, and symmetry analysis, and of fundamental importance for object recognition in image processing, is the \emph{equivalence problem}, that is, determining when two objects in a space can be mapped to each other by a transformation belonging to a prescribed group or pseudo-group action.
\'Elie Cartan's solution to the equivalence problem for submanifolds under transformation groups, \cite{Cartanequiv}, is based on the functional interrelationships among the associated differential invariants.  Cartan's result was reformulated through the introduction of the \emph{classifying submanifold}, \cite E, subsequently --- motivated by the extensive range of applications in image processing --- renamed the \emph{differential invariant signature}, \cite{COSTH}.  The signature of a submanifold is parametrized by a finite number of fundamental differential invariants\footnote{Identification of the required differential invariants can be facilitated and systematized through the equivariant moving frame calculus and, specifically, the recurrence formulae, \cite{FOmcII, Mansfield, Omfl}.  The case of curves is straightforward.} and one proves that two sufficiently regular submanifolds are locally equivalent under a group transformation if and only if their signatures are identical.  The symmetries of a submanifold can also be classified by the dimension and, in the case of discrete symmetries, the index of its associated signature. 

Differential invariant signatures of families of curves were used  in \cite{BurdisKogan, BuKoHo} to establish a novel algorithm  for solving the object-image correspondence problem for curves under projections.
  Extensions of the method to signatures parametrized by joint invariants and joint differential invariants, also known as semi-differential invariants, \cite{MPVO}, can be found in \cite{Oji}.  A wide range of image processing applications includes jigsaw puzzle assembly, \cite{HOpuzzle}, recognition of DNA supercoils, \cite{ShLlP}, distinguishing malignant from benign breast cancer tumors, \cite{GrimShak},
recovering structure of three-dimensional objects from motion, \cite{BazBou}, classification of projective curves in visual recognition,
\cite{HannHick}, and construction of integral invariant signatures for object recognition in 2D and 3D images, \cite{FengKoganKrim}. Further applications of the moving frame-based signatures include classical invariant theory, \cite{BOpol, Kogan-thesis, KoganMM,I}, symmetry and equivalence of polygons and point configurations, \cite{Boutinpol,KemperBoutin}, 
geometry of curves and surfaces in homogeneous spaces,
with applications to Poisson structures and integrable systems, \cite{MBp,MBhs}, the design and analysis of geometric integrators and symmetry-preserving numerical schemes, \rf{KimO,MKLun,RebVal}, the determination of Casimir invariants of Lie algebras and the classification of subalgebras, with applications in quantum mechanics, \cite{BoPaPo}, and many more.

In our analysis of the behavior of invariants under surjective maps, we will concentrate on finite-dimensional Lie group actions, although our analysis can, in principle, be extended to infinite-dimensional Lie pseudo-groups, using the techniques developed in \cite{OP2,OP3}.  We will distinguish between projectable group actions, in which the group transformations respect the surjective map's fibers, and the more general non-projectable actions.   In the former case, there is a naturally defined action of a certain quotient group on the image manifold, and we are able to directly relate the differential invariants and, hence, the differential invariant signatures of submanifolds and their projected images.    

However, in the image processing applications we are primarily interested in the case when only a (fairly large) subgroup of the full transformation group acts projectably, and thus we need to extend our analysis to non-projectable group actions.  In this situation, one distinguishes a projectably acting subgroup, along with its corresponding projected action and invariants on the image manifold. 
Then the full transformation group will have a residual effect on the image invariants and signatures, which are no longer fully invariant, and hence the comparison of the projected images must take this into account.
  For example, in the case of central  projection based at the origin, from three-dimensional space to the two-dimensional plane, the ``centro-affine'' action of the general linear group $\GL3$ is projectable, and this leads to
  our formulas relating centro-affine differential invariants to projective differential invariants on the image curve.  On the other hand, translations are not projectable, and thus have a residual effect on the projective invariants that will be explicitly characterized.

\Section2{Projectable actions: \hfill \break \hglue 1.4in  invariants of objects and images}

In this section, we consider projectable actions of a  Lie group $G$ on a manifold $M$ meaning that they respect the fibers of a surjective map $\P\colon M\to N$.  A projectable action on $M$ induces a natural action on $N$.  We establish an   isomorphism between the algebra of differential invariants for submanifolds on $N$  and  the algebra of fiber preserving (gauge) differential invariants on $M$. This isomorphism allows us to express invariants of the image  of a submanifold $S\subset   M$  in terms of the invariants of $S$.  Since the equivariant moving frame method  \cite{FOmcII,OP2} provides a powerful and algorithmic tool for constructing invariant objects,  we are able to explicitly determine how invariant functions and invariant differential operators on $N$, obtained via this method, are  related to their counterparts on $M$.  

In this paper, all objects --- manifolds, submanifolds, Lie groups, maps, differential forms, etc. --- are assumed to be smooth, meaning of class $\Cci$.

\subsection{Transformation groups}

Let $G$ be a Lie group (or, more generally, a Lie pseudo-group, \cite{OP3}) acting on a smooth manifold $\M $. In this paper, many of the actions that we consider are \emph{local actions}, although we will usually omit the word \emph{local} when we describe them.

\begin{definition} \label{iso} \rm
The \emph{isotropy subgroup} of a subset $S\subset \M$ consists of the group elements which fix it: 
$$\GS = \set{g\in G}{g\cdot S = S}.$$
The \emph{global isotropy subgroup} of $S$ consists of the group elements which fix \emph{all} points in $S$\/:
$$
\giso S = \bigcap_{z\in S} G_z = \set{g\in G}{g\cdot z = z \xbox{ for all } z \in S}.
$$
\end{definition}

In particular, the global isotropy subgroup $\giso \M$ of $\M$ is a normal subgroup of $G$.  The action of $G$ is \emph{effective} if and only if $\giso \M = \{e\}$ is trivial.  More generally, the action of $G$ induces an equivalent effective action of the quotient group $G/\giso \M$ on $\M$, \rf E.

The following group actions will play a prominent role in our examples.  Each matrix $A\in \GL n$ produces an invertible linear transformation  $\z \mapsto A\,\z$ for $\z \in \R^n$.    More generally, we consider the action of the \emph{affine group} $\AG n = \GL n \ltimes \R^n$ given by $\z \mapsto A\,\z + b$ for $A \in \GL n$, $b \in \R^n$.  This action forms the foundation of affine geometry, and, for this reason, the previous linear action of $\GL n$ is sometimes referred to as the \emph{centro-affine group}, underlying centro-affine geometry, \cite{Cruceanu, Ocentro}. We also consider the action of the projective group $\PGL n=\GL n/ \set{ \lambda \, \mathrm I}{ 0 \ne \lambda \in \R}$ on the projective space $\RP^{n-1}$ along with its local, linear fractional action on the dense open subset $\R^{n-1} \subset \RP^{n-1}$ obtained by omitting the points at infinity.
 
\medskip
 
\emph{Warning}: In many references, ``affine geometry'' really refers to ``equi-affine geometry'' whose underlying transformation group is the \emph{special affine} or \emph{equi-affine group} $\SA n = \SL n \ltimes \R^n$ consisting of oriented volume-preserving transformations: $\z \mapsto A\,\z + b$ with $\det A = 1$.  We also use the term \emph{centro-equi-affine geometry} to indicate the linear volume-preserving action, $\z \mapsto A\,\z$ with $\det A = 1$, of the special linear group $\SL n$ on $\R^n$. 

\medskip

\subsection{Projectable actions}

Our principal object of study is the behavior of group actions under a surjective map $\P \colon \M \to \N $ of constant rank from a manifold $\M$ onto a manifold $\N$ of lower dimension: $\n = \dim \N < \m = \dim \M$.  Given $\w\in \N $, let $\F_\w =\P ^{-1}\{\w\} \subset \M $ denote its preimage, called the \emph{fiber} of $\P $ over $\w$. In many examples, $\M $ is, in fact, a fiber bundle over $\N $, but we do not require this in general.  The kernel of the map's surjective differential $d \P \colon T\M \to T\N$ is  the tangent space to the fiber: $T \F_\w\at z  = \ker d\Pi \at z \subset T\M\at z$, where $\w = \Pi(z)$. 

To begin with, we will consider group actions that are compatible with the surjective map in the following sense.


\begin{definition}\label{def-proj}  \rm A group action of $G$ on $\M $ is called \emph{projectable} under the surjective map $\P \colon \M \to \N $ if, for all $\w\in \N $ and  for all $g\in G$, there exists $\barr \w\in \N $ such that $g \cdot \F_\w =\F_{\barr \w}$.
\end{definition}

In other words, the action of $G$  is projectable if and only if  it maps fibers to fibers.
In this case, it is clear that the induced map $\w \mapsto \barr \w = g\cdot \w$ is a well-defined action of $G$ on $\N $, satisfying 
\beq\label{eq-star} g\cdot \w = \P (g \cdot \F_\w).\eeq
As above, we define the \emph{global isotropy subgroup}
\beq \label{eq-gis}
\eeqo{\GN =\set{g\in G}{g\cdot \w =\w \xbox{ for all } \w\in \N }\\
= \set{g\in G}{g \cdot \F_\w =\F_\w \xbox{ for all } \w\in \N }=\bigcap_{\w\in \N }   \Gw,}
\eeq
where 
$$\Gw =\set{g\in G}{g\cdot \w =\w}=\set{g\in G}{g \cdot \F_\w =\F_\w}$$
is the \emph{stabilizer} or \emph{isotropy subgroup} of the point $\w\in \N $. 
The action of $G$ on $\N $ induces an equivalent, effective action of the quotient group 
$$[G] =G/\GN$$
 on $\N $.  We use the notation $[g] =g\,\GN  \in [G]$ to denote the element of the quotient group corresponding to $g \in G$.

By a \emph{$G$-invariant} function, we mean a real-valued function $J \colon \M  \to \R$ that is unaffected by the group action, so $J(g \cdot \z) = J(\z)$ for all $g \in G$ and all $\z \in \dom J$ such that $g\cdot \z \in \dom J$.  (Our notational conventions allow $J$ to only be defined on an open subset $\dom J \subset \M $.  Also, if the action of $G$ is local, one only requires the invariance condition to hold when $g \cdot z$ is defined and in the domain of $J$.)  Clearly a function is $G$-invariant if and only if it is constant on the orbits of $G$.  In particular, when $\M$ is connected and $G$ acts transitively, then there are no non-constant invariants.
On occasion, one relaxes the preceding definition, by only imposing invariance for group elements sufficiently close to the identity, leading to the concept of a \emph{local invariant}.
The correspondence between $[G]$-invariant functions on $\N $ and $G$-invariant functions on $\M $ follows straightforwardly from \eq{eq-star}.

\begin{theorem}\label{th-inv} Let $\P \colon \M \to \N $ be a surjective map, and suppose that $G$ acts projectably on $\M $. 
If $I\colon \N \to \R$ is a $[G]$-invariant function on $\N $, then $\widehat I= I\comp \P \colon \M \to \R$ is a $G$-invariant function on $\M $.  Conversely, any $G$-invariant function $\widehat I\colon \M \to \R$ that is constant on the fibers of $\P$ induces a  $[G]$-invariant function $I\colon \N \to \R$ such that  $\widehat I = I\comp \P$.
\end{theorem}

\subsection{Submanifolds}

Let us now investigate how a projectable group action affects submanifolds and their jets.  We will assume that the submanifolds are immersed, although in many situations one restricts attention to embedded  submanifolds.  Throughout, we fix the dimension $p$ of the submanifolds under consideration, and assume that $1 \leq p < \n = \dim \N  < \m = \dim \M $.

\begin{definition} \label{def-Preg} \rm
A $p$-dimensional submanifold $S \subset \M $ is called \emph{$\P $-regular} if its projection $\P (S)$ is a smooth  $p$-dimensional  submanifold of $\N $. 
\end{definition}

Because we are allowing immersed submanifolds, the following transversality condition is both necessary and sufficient for $\P $-regularity.

\begin{proposition}  
A submanifold $S \subset \M $ is {$\P $-regular} if and only if it intersects the fibers of $\P$ transversally\/{\rm :} 
\beq\label{tc}
T_\z S\,\cap\, \ker d\P\:\at \z  =\{0\} \qbox{ for all }  z \in S.
\eeq
\end{proposition}

Because condition \eq{tc} is local, it is a necessary but not sufficient condition for the image $\P (S)$ of an embedded $p$-dimensional submanifold $S \subset\M$ to be an embedded $p$-dimensional submanifold of $\N$. For example, many embedded curves in $\R^3$, \eg nontrivial knots, can only be projected to plane curves with self-intersections.

Suppose we adopt local coordinates $\z = \psups \z\m $ on $\M $ and $\w = \psups \w\n$ on $\N $.   In terms of these, the surjective map $\w = \P(z)$ has components 
$$\w^i = \P^i\psups \z\m , \qquad \rg in.$$
  If the submanifold $S \subset \M $ is (locally) parametrized by $\z = \z(t) = \z\psups tp$, then its tangent bundle $TS$ is spanned by the basis tangent vectors
$$\vb_i = \Sum a\m \pd{\z^a}{t^i} \pdo{\z^a}\,,\qquad \rg ip.$$
Since
$$d\P(\vb_i) = \Sum k\n \Sum am \pd{\z^a}{t^i} \pd{\P^k}{\z^a} \pdo{\w^k}\,,$$
the transversality condition \eq{tc} holds if and only if the associated $p \times n$ coefficient matrix has maximal rank:
\Eq{maxrank}
$$\rank \Pa{\Sum a\m \pd{\z^a}{t^j} \pd{\P^k}{\z^a}} = p.$$

Often, it will be useful to split the coordinates on $\M$, setting  $\z = (\sups xp,\sups u{m-p})$,  in which the $x$'s will play the role of independent variables and the $u$'s dependent variables. A $p$-dimensional submanifold $S$ that is transverse to the \emph{vertical fibers} $\{x = c\}$, for $c$ constant, can be locally identified as the graph of a function:  $S = \{ (x,u(x))\}$.  Hence, its tangent space $TS$ is spanned by the tangent vectors
\Eq{vi}
$$\vb_i = \pdo{x^i} + \Sum \alpha {m-p} \pd{u^\alpha }{x^i} \pdo{u^\alpha }\,,\qquad \rg ip.$$
In this case, the coefficient matrix \eq{maxrank} reduces to the  $p \times n$ \emph{total derivative matrix}
\beq\label{DP}
D\P = \Pa{D_i\P^k} = \Pa{\pd{\P^k}{x^i} + \Sum \alpha {m-p} \pd{u^\alpha }{x^i}\pd{\P^k}{u^\alpha }} \qbox{where} \rg ip, \ \ \rg k\n,
\eeq
 which, to ensure $\P$-regularity, is again required to have maximal rank: 
\beq\label{DPmax}
\rank D\P = p.
\eeq

\subsection{Jets and differential invariants}

Given $0 \leq k \leq \infty$, let $\Jk(\M ,p)$ be the $k$-th order \emph{extended jet bundle} consisting of equivalence classes of $p$-dimensional submanifolds of $\M $ under the equivalence relation of $k$-th order contact, \rf{O}.  In particular $\J0(\M ,p) = \M $.  When $l \geq k\geq 0$, we use $\pi^l_k \colon \Jl(\M ,p) \to \Jk(\M ,p)$ to denote the standard projection.  

Given a surjective map $\P \colon \M \to \N $, let $\Jtk(\M ,p)\subset \Jk(\M ,p)$ be the open dense subset consisting of $k$-jets of $\P $-regular submanifolds, i.e.~those that satisfy the transversality condition \eq{tc}, or, equivalently, in local coordinates, condition \eq{DPmax}.  Note that transversality defines an open condition on the first order jets, so that $\Jtk(\M ,p) = (\pi^k_1)^{-1} \Jt1(\M ,p)$.
Let $\Pk \colon \Jtk (\M ,p)\to  \Jk (\N ,p)$ denote the induced surjective map on $p$-dimensional submanifold jets, that maps the $k$-jet of a transversal submanifold $S$ at a point $\z \in S$ to the $k$-jet of its image $\P(S)$ at $\w = \P(\z)$.  In other words, if $\zk = \jk S \at \z$ then $\wk = \Pk(\zk) = \jk \P(S) \at {\P(\z)}$.  The fact that $\P$ preserves the condition of $k$-th order contact between submanifolds (which is a simple consequence of the chain rule), means that $\Pk$ is well-defined on $\Jtk (\M ,p)$.

Given the action of $G$ on $\M $, there is an induced action on $p$-dimensional submanifolds, and hence on the jet space $\Jk(\M ,p)$,  called the $k$-th order \emph{prolonged action} and denoted by $\Gk$.  Namely, if $\zk = \jk S \:\at \z \in \Jk(\M ,p)$ is the jet of a submanifold at $z \in S \subset \M $, and $g \in G$, then 
$\gk \cdot \zk = \jk (g \cdot S) \at{g \cdot z}$.  Because diffeomorphisms preserve $k$-th order contact, the action is independent of the choice of representative submanifold $S$, \cite E.

The  action of the quotient group $[G]$ on $\N$ similarly induces a prolonged  action, denoted by $[G]\ps k$, on its $k$-th order submanifold jet bundle $\Jk(\N ,p)$. 
It is not hard to see that the jet bundle projection $\Pk$ respects the prolonged group actions of $\Gk$ on $\Jtk(\M ,p)$ and $[G]\ps k$ on $\Jk(\N ,p)$.  In other words, 
\beq\label{gnzn}
[g]\ps k \cdot \Pk(\zk) = \Pk(\gk \cdot \zk),
\eeq
provided both $\zk ,\ \gk \cdot \zk \in \Jtk(\M ,p)$. Indeed, to verify \eq{gnzn}, just set $\zk = \jk S \at \z$ for some submanifold $S \subset \M $ and use the preceding identifications.

A real-valued function $\Fh \colon \Jk(\M ,p) \to \R$ is called a  \emph{differential function} of order $k$.  (As before, our conventions allow functions, differential forms, etc., to only be defined on open subsets, so $\dom \Fh \subset \Jk(\M ,p)$.)   A \emph{differential invariant} is a differential function $\Ih \colon \Jk(\M ,p) \to \R$ that is invariant under the prolonged group action: $\Ih(\gk \cdot \zk) = \Ih(\zk)$ whenever both $\zk$ and $\gk \cdot \zk \in \dom \Ih$. In view of \eq{gnzn}, Theorem \ref{th-inv} immediately establishes a correspondence between differential invariants on $\N $ and those on $\M $ under a $\P$-projectable group action.

\begin{theorem} \label{th-dinv} 
Let $\P \colon \M \to \N $ be a surjective map and let $G$ act projectably on $\M $.  If $I\colon \Jk(\N ,p)\to\R$ is a differential invariant for the prolonged action of $[G]$ on $\N $, then $\Ih = I \comp \Pk \colon \Jtk(\M ,p)\to\R$ is a differential invariant for the prolonged action of $G$ on $\M $, with domain $\dom \Ih = \P^{-1}(\dom I)$.
 \end{theorem}
 
 Of course, not every differential invariant on $\M $ arises in this manner.  Indeed, $\Ih = I \comp \Pk$ for some differential invariant $I$ on $\N $ if and only if $\Ih$ is constant along the fibers of the jet projection $\Pk$.  Such differential invariants will be called \emph{gauge invariants}, and we investigate their properties in Section~\ref{gauge}.

\subsection{Invariant differential forms and differential operators}
\label{sect-inv-diff-forms}
Turning to differential forms, we assume the reader is familiar with the basic variational bicomplex structure on jet space, \cite{A,FOmcII,KOivb}.  As usual, for certain technical reasons, it is preferable to work on the infinite jet bundle even though all calculations are performed on jet bundles of finite order.  

As above, we introduce local coordinates $\z = (x,u) = (\sups xp,\sups u{m-p})$ on $\M$, where the $x$'s represent independent variables. The differential one-forms on $\Ji(\M ,p)$ then split into \emph{horizontal forms}, spanned by $\sups{dx}p$, and \emph{contact forms}, which are annihilated when restricted to a prolongation of any $p$-dimensional submanifold on $M$.   
  The induced splitting of the differential $d = \dH + \dV$ into horizontal and vertical (contact) components endows the space of differential forms on $\Ji(\M ,p)$ with the powerful \emph{variational bicomplex} structure, playing important role in geometric study of differential equations, variational problems, conservation laws, characteristic classes, etc. 

\Rmk While the contact component is intrinsic, the horizontal forms, and hence the induced splitting, depend upon the choice of independent variable local coordinates.  A more intrinsic approach is based on filtrations and the $\cal C$ spectral sequence, \cite{VinoI,VinoII}; however, this extra level of abstraction is unnecessary in what follows. 

We use $\piH$ to denote the projection of a one-form onto its horizontal component, so $\dH \Fh = \piH(d\Fh)$ for any differential function $\Fh \colon \Ji(M,p) \to \R$. 
The symbol $\equiv$ is used to indicate equivalence modulo the addition of contact forms, so that $\omega \equiv \piH(\omega )$; thus, we mostly only display the horizontal components of the pulled-back forms.

Let $\P \colon \M \to \N $ be a surjective map.  
Let $\sups yp$ denote a subset of the local coordinates $v^1,\dots,v^n$ that we consider as independent variables.
The corresponding horizontal forms $\sups{dy}p$ on $\Ji(\N ,p)$ are pulled-back by $\w = \P(\z) = \P(x,u)$ to 
\beq\label{Pdyk}
\P\Upstar(dy^k) \equiv \Sum ip (D_i\P^k) \, dx^i, \qquad \rg kp.
\eeq
Thus,  the pulled-back one-forms \eq{Pdyk} will form a basis for the space of horizontal one-forms on $\Ji(\M ,p)$ provided the $p \times p$ minor consisting of the first $p$ columns of the full $p \times n$ total derivative matrix $D\P$ given in \eq{DP} 
is non-singular:
\beq\label{D0P}
\det D_0\P \ne 0, \qbox{where} D_0\P = \Pa{D_i\P^k}, \quad \rg{i,k}p.
\eeq
Observe further that our $\P$-regularity condition \eq{DP} implies that some $p \times p$ minor of $D\P$ is non-singular, and hence, locally, one can always choose a suitable set of local coordinates on $\N $ such that the non-singularity condition \eq{D0P} holds. 

It is well known that the algebra of differential invariants of a Lie transformation group, \cite{FOmcII,E}, or (modulo technical hypotheses) a Lie pseudo-group, \cite{KruLycLT,OP3}, is generated from a finite number of low order \emph{generating differential invariants} through successive application of the operators of invariant differentiation.  The construction of the generating differential invariants, the invariant differential operators, and the identities (syzygies and recurrence relations among them) can be completely systematized through the symbolic calculus provided by the equivariant method of moving frames, \cite{FOmcII, KOivb, Mansfield,OP2}.  In particular, the moving frame invariantization process allows one to construct a \emph{contact-invariant horizontal coframe}, that is, a linearly independent set of $p$ horizontal contact-invariant one-forms
\beq\label{ih1df}
\omega ^i = \Sum jp Q^i_j(\wk) \, dy^j , \qquad \rg ip,
\eeq
on $\Jk(\N ,p)$, where $0 \leq k < \infty$ is the order of the equivariant moving frame map.  The term ``contact-invariant'' means that each one-form is invariant under prolonged group transformations modulo contact forms, i.e., for each $[g] \in [G]$, each $\omega ^i$ agrees with the horizontal component of its pull-back:
$$[g]\ps k\,\Upstar \,\omega ^i \equiv \omega ^i, \qquad \rg ip.$$
   Each $\omega ^i$ is, in fact, the horizontal component of a fully $[G]$-invariant one-form, whose additional contact component, which will not be used here, can also be explicitly constructed via the method of moving frames, \cite{FOmcII,KOivb}.  For instance, in the case of curves, so $p=1$, under the action of the Euclidean group, the contact-invariant one-form is the standard arc length element $\omega = ds$, which can be identified as the horizontal component of a fully invariant one-form.  

Given the  horizontal coframe \eq{ih1df}, the corresponding dual \emph{invariant differential operators} $\subs{\CD}p$ are defined so that
\beq\label{dH}
d_H F = \Sum jp (\CD_j F)\, \omega ^j
\eeq
for any differential function $F \colon \Ji(\N ,p) \to \R$. In particular, if $I$ is a differential invariant, so are its derivatives $\CD_j I$ for $\rg ip$, and hence, by iteration, all higher order derivatives $\CD_J I = \CD_{j_1} \cdots \CD_{j_k} I$, $k = \#J \geq 0$, are differential invariants as well. For example, in the case of the Euclidean group acting on curves, the dual to the contact-invariant arc length one-form $\omega = ds$ is the total derivative with respect to arc length, denoted $\CD = D_s$.  Applying $\CD$ to the basic curvature differential invariant $\kappa $ produces a complete system of differential invariants $\kappa, \ \kappa _ s = \CD \kappa, \ \kappa _{ss} = \CD^2 \kappa, \ \ldots\>$, meaning that any other differential invariant can be written (locally) as a function thereof.

Using the surjective map $\Pk$ to pull-back the horizontal one-forms \eq{ih1df} produces, by a straightforward generalization of Theorem \ref{th-dinv}, a system 
\beq\label{lih1df}
\lomega ^i = \piH\bigl[\,\Pks \,\omega^i\,\bigr] = \Sum kp P^i_k(\zk) \, dx^k  , \qquad \rg ip,
\eeq
of $G$-contact-invariant  horizontal one-forms on $\Jk(\M ,p)$, whose coefficients $P^i_k(\zk)$ can be readily constructed from the local coordinate formulas for $\P$, the horizontal one-forms $\omega ^i$, along with \eqf{Pdyk}.  Under the non-singularity condition \eq{D0P}, the resulting one-forms are linearly independent, and hence determine dual invariant total differential operators $\subs {\CE}p$ on $\Ji(\M ,p)$, satisfying
\beq\label{dHM}
d_H \Fhat = \Sum jp (\CE_j \Fhat)\, \lomega ^j
\eeq
for any differential function $\Fhat \colon \Ji(\M ,p) \to \R$.  

Summarizing the preceding discussion: 

\begin{theorem} \label{th-DI} 
Let $\P \colon \M \to \N $ be a surjective map.  Suppose that the action of $G$ on $\M $ is $\P $-projectable. Let $\sups \omega p$  be a $[G]$-contact-invariant horizontal coframe on $\Ji(\N ,p)$, and let $\subs \CD p$ be the dual invariant differential operators. 
For $\rg ip$, let $\lomega^i$ be the horizontal component of the pulled-back one-form $\P\Upstar(\omega ^i)$.  Then, provided the non-singularity condition \eq{D0P} holds, $\sups \lomega p$ form a  $G$-contact-invariant horizontal coframe on an open subset of $\Ji(\M ,p)$. Let $\subs {\CE}p$ be the dual invariant differential operators, satisfying \eq{dHM}.  If $F\colon \Jk(\N ,p) \to \R$ is any differential function on $\N $ and $\Fh = F\comp \Pk \colon \Jk(\M ,p) \to \R$ the corresponding differential function on $\M $, then
\beq\label{CDF}
\CE_i \Fh=\CE_i(F\comp \Pk )=(\CD_i F)\comp \P ^{(k+1)} = \widehat{\CD_i F}.
\eeq
\end{theorem}

The proof of the final formula \eq{CDF} follows from the fact that, since $\Pks$ maps contact forms to contact forms,
$$\piH \bigl[\,\Pks(d_H \Omega )\,\bigr] = d_H\bigl[\,\Pks\,\Omega\,\bigr ]$$
for any differential form $\Omega $ on $\Jk(\N ,p)$.  Taking $\Omega = F$ reproduces \eq{CDF}.
In particular, if $I\colon \Jk(\N ,p) \to \R$ is a differential invariant on $\N $ and 
$$\Ih = I\comp \Pk \colon \Jk(\M ,p) \longrightarrow  \R$$
is the induced differential invariant on $\M $, then their invariant derivatives are directly related:
\beq\label{CDI}
\CE_i \Ih =\CE_i(I\comp \Pk )=(\CD_i I)\comp \P ^{(k+1)}= \widehat{\CD_i I}.
\eeq
Thus, the prolongations of the surjective map $\P$ provide an explicit isomorphism between the algebra of differential invariants on $\N $ and the subalgebra of fiber-wise constant differential invariants on $\M $.

\subsection{Gauge invariants}
\label{gauge}


In this section, we investigate the structure of the aforementioned subalgebra of fiber-wise constant differential invariants on $\M $ in further detail.
Although we are not necessarily dealing with fiber bundles, we will adapt standard terminology to this situation.  Define the \emph{gauge group} of the surjective map $\P$  to be the pseudo-group  
 \beq\label{eq-GP}
 \GP=\set{\varphi\in\Diff_{\rm loc}(\M)}{\varphi(\F_\w\cap \dom \varphi ) \subset \F_\w \ \xbox{for all} \ \w\in \N} ,
 \eeq
consisting of all local diffeomorphisms of $\M$ that fix the fibers of $\P$.  (By a \emph{local diffeomorphism}, we mean a smooth, locally defined, one-to-one map with smooth inverse.)
Clearly $\GP$ acts transitively on each fiber.  Indeed, since $\P$ is a submersion,  around each point $z_0 \in M$ there exist local coordinates $\z =  (\w,\y) = (\sups \w n,\sups \y{m-n})$ such that $\P(\z) = \w = (\sups \w n)$ provide the induced local coordinates on $\N$. We will call such coordinates \emph{$\P$-canonical}.

In $\P$-canonical coordinates,  the elements of $\GP$ take the form $(\w,\y) \longmapsto (\w, \psi (\w,\y))$, where, for each fixed $\w$, the map $\psi _\w(\y) = \psi (\w,\y)$ is a local diffeomorphism of $\R^{m-n}$.  Given $1 \leq p<n$, any $\P$-regular $p$-dimensional submanifold $S\subset \M$   can be parametrized by a subset, $\sups xp$,  of the  $\w$-coordinates; we write the remainder of $\w$-coordinates as $\v = \psups \v{n-p}$, so that, by suitably relabeling, $\w = (x,\v)$. Then $x$'s will play the role of independent variables, while $\v$'s and $\y$'s play the role of dependent variables on $\M$. 
At the same time,  $x$'s  and $\v$'s will play the roles of independent and dependent variables, respectively, on $\N$.    

The fibers of $\Pk\colon \Jtk (\M ,p)\to  \Jk (\N ,p)$ are parametrized by the induced jet coordinates $\y^\alpha _J$, where $\alpha =1,\dots,m-n$, and $J$ is a symmetric multi-index of order $\leq k$. 
Clearly, the prolonged action of $\GP\ps k$ on the jet space $\Jk(\M,p)$ is also transitive on the fibers of $\Pk$.

We can thus identify the fiber-wise constant (differential) invariants on $\M$ with the (differential) invariants of the semi-direct product pseudo-group $G \ltimes \GP$.  We will call these  \emph{gauge invariants} and \emph{gauge differential invariants} for short.    
 
 \begin{proposition} The algebra of gauge differential invariants coincides with the algebra of differential invariants for the action of $G \ltimes \GP$.
 \end{proposition}

In $\P$-canonical coordinates, a projectable action of $G$ on $M$ takes the form
$$(\w,\y)\longmapstox (\phi(\w), \chi(\w,\y)).$$
The projected action of  $[G]=G/G_N$ on $N$ is then given by $\w=(x,\v)\mapsto \phi(\w)$.
 We observe that  the prolongation of the $\GP $-action  leaves the jet coordinates $(x,\v^\beta _K)$ invariant and, moreover, its differential invariants are independent of the $\y^\alpha _J$ coordinates.  Thus, in the canonical coordinates, the isomorphism between  the fiber-wise constant differential invariants under the prolonged action of $G$ on $M$ and the differential invariants under the prolonged action of $[G]$ on $N$ becomes transparent. 
 
\Rmk While the general expressions simplify when written in canonical coordinates, in  examples, this may not be practical because the explicit formulas for the group action, differential invariants, etc.~may be 
unavailable or just too complicated to work with. Furthermore, canonical coordinates may have a restricted domain of definition, and hence less suitable for visualization and analysis of geometric objects.

 
 \begin{example} \label{ex-cp} 
 \rm

Let $\M = \set{(x,y,z) \in \R^3}{z\neq0}$. Consider the surjective map
 \beq\label{eq-P0Z}
 (X,Y)=\Po(x,y,z)=\left(\frac {x}{ z}, \frac {y}{ z}\right), \qquad  (x,y,z) \in \M,
 \eeq
 onto $\N = \R^2$. Note that we can identify the map $\Po$ with central projection, centered at the origin, from $\M $
 to the plane $\N \simeq \R^2$ defined by $z=1$.
 The fibers of $\Po$ are the rays in $M$ emanating from the origin. 

Observe that 
\beq\label{eq-c-can}\qeq{X = x/z,\\Y = y/z, \\Z=z,}\eeq 
form canonical coordinates for $\Po$ on $\M$, in which $\GP$ consists of all local diffeomorphisms of the form  $(X,Y,Z) \longmapsto (X, Y, \varphi (X,Y,Z))$ 
or, equivalently, in the  original coordinates, $$(x,y,z) \longmapstox (\psi(x,y,z)\, x,\, \psi(x,y,z)\,y,\,\psi(x,y,z)\,z),$$ where $\psi(x,y,z) = \varphi(x/z,y/z,z){/z }$.

The local\footnote{The action is local because of the  restriction  $z\ne 0$.} 
centro-affine action of $\GL3$ on $\M$ is $\Po$-projectable.
In $\Po$-canonical coordinates, it takes the form 
\beq \label{G-canonical}\hskip0pt
(X,Y,Z)  \longmapsto  \Pa{\frac{a_{11}\,X+a_{12}\,Y+a_{13}}{a_{31}\,X+a_{32}\,Y+a_{33}},\ \frac{a_{21}\,X+a_{22}\,Y+a_{23}}{a_{31}\,X+a_{32}\,Y+a_{33}}, \  (a_{31}\,X+a_{32}\,Y+a_{33})\,Z}\!,\hskip-15pt
\eeq
where $A = (a_{ij}) \in \GL3$. The global isotropy group 
$$G_N = \set{ \lambda \, \mathrm I}{ 0 \neq \lambda \in \R}$$ 
consists of the uniform scalings, i.e.~nonzero multiples of the identity matrix, and hence the quotient group is the \emph{projective linear group} $[G] = G/G_N = \PGL3$. The induced action of $[G] = \PGL3$ on $\N$ coincides with the usual linear fractional action
\beq \label{eq-proj-l}(X,Y)  \longmapstox \Pa{\frac{a_{11}\,X+a_{12}\,Y+a_{13}}{a_{31}\,X+a_{32}\,Y+a_{33}},\ \frac{a_{21}\,X+a_{22}\,Y+a_{23}}{a_{31}\,X+a_{32}\,Y+a_{33}}}
\eeq  
on the projective plane. We regard $X$ as the independent variable, and $Y,Z$ as dependent variables on $M$, with $Y$ also serving as the dependent variable on $N$. 

The algebra of fiber-wise constant $G$-differential invariants on  $\Jk(\M,1)$ coincides  with the algebra of $G\ltimes \GP$-differential invariants.
Since $\GP$ leaves $X, Y$ as well as the jet variables $Y_X, Y_{XX},\ldots$ invariant, and does not admit any invariants depending on $Z,Z_X, Z_{XX}, \ldots$,  the algebra of  $G\ltimes \GP$-differential invariants on $\M$ is isomorphic to the algebra of differential invariants for the standard projective action of $\PGL3$ on $\N$.  See Example \ref{ex-cp2} below for explicit formulas.

 \end{example}

\subsection{Cross-sections and invariantization}
\label{sect-Pi-iota}

The construction of an equivariant moving frame  relies on the choice of  a cross-section to the (prolonged) group orbits, \cite{FOmcII,OP2}. In this section,  we investigate what happens when we choose cross-sections on $\M$ and $\N$ that are compatible under the surjective map $\P$.

As before, let $G$ be a Lie group acting on the manifold $\M$.  Let $\orbit_z$ denote the orbit through the point $z \in \M$.

 \begin{definition} \rm A submanifold $\K  \subset \M$ is called  \emph{a local cross-section} to the group action if there exists an open subset $\U\subset M$, called the \emph{domain of the cross-section}, such that, for each $\z \in \U$,  the connected component $\orbit^0_z$ of  $\orbit_z\cap\,\U $ that contains $z$  intersects $\K$ transversally  at a single point,   
 so $\orbit^0_z \cap \K = \bc{z_0}$ and  $T\K|_{z_0}\oplus T\orbit_z|_{z_0}=T\M|_{z_0}$.\end{definition}

Let $s$ denote the maximal orbit dimension of the $G$-action on $M$. If a point $z$ belongs to an orbit of dimension $s$, then the Frobenius Theorem, \cite{E}, implies the existence of a local cross-section $\K$, of codimension $s$, whose domain includes $z$. While the definition of a cross-section allows $s < r = \dim G$, the construction of a locally equivariant moving frame map $\rho \colon \U \to G$ requires that the group act locally freely, which is equivalent to the requirement that $s = r$.
 
Let $\cFU$ denote the algebra of all smooth real-valued functions $F \colon \U \to \R$, and $\cFUG$ the subalgebra of all locally  $G$-invariant functions.  Note that each locally invariant function $I \in \cFUG$ is uniquely determined by its values on the cross-section, namely $I \mid \K$, since, by invariance, $I$ is constant along each orbit.
Thus, the cross-section $\K$ serves to define an \emph{invariantization map} $\inv\colon \cFU \to \cFUG$, which maps a function $F$  on $\U$ to the unique locally invariant function $\inv (F)$ that has the same values on the cross-section: 
$$\inv (F) \mid \K=F\mid \K.$$
This immediately implies that the invariantization map  preserves all algebraic operations.  Moreover, if $I$ is an invariant, then $\inv(I) = I$, which implies that $\inv \comp \inv = \inv$.  In other words, $\inv \colon \cFU \to \cFUG$  is an algebra morphism that canonically projects functions to invariants.

In local coordinates $z = \psups zm$, invariantization maps the coordinate function $z^i$ to the fundamental invariant $I^i = \iota(z^i)$.   The $r = \dim G$ functions $\subs Fr$ that serve to define the cross-section, $\K = \{F_j(z) = c_j,\ \rg j r\}$, have constant invariantizations, $\inv(F_j) = c_j$, and are known as the \emph{phantom invariants}.  This leaves $m-r$ functionally independent invariants, which can be selected from among the fundamental invariants $I^i$. In particular if one uses a coordinate cross-section, say $\K = \{z^j = c_j,\ \rg j r\}$, then the first $r$ fundamental invariants are the constant phantom invariants: $I^1 = \inv(z^1) = c_1, \ \ldots\ , I^r = \inv(z^r) = c_r$, and the remainder form a complete system of functionally independent invariants $I^{r+1} = \inv(z^{r+1}),\ \ldots \ , I^m = \inv(z^m)$, meaning that any other invariant can be expressed in terms of them.  Indeed, invariantization of a function  is done by simply replacing each variable $z^i$ by the corresponding fundamental invariant:
 \beq\label{invz} \inv\bigl[F\psups zm\bigr] = F\psups Im.\eeq 
In particular, if $J = \inv(J)$ is any invariant, then we can immediately rewrite it in terms of the fundamental invariants by simply replacing each variable by its invariantization:
 \beq\label{Replacement} J\psups zm = J\psups Im.\eeq 
This simple, but remarkably powerful result is known as the Replacement Theorem, \cite{FOmcII}. Assuming local freeness, the invariantization process can also be applied to differential forms, producing the corresponding invariant differential forms, their dual invariant differential operators, and, more generally, vector fields, all of whose explicit formulae can be obtained via the equivariant moving frame map $\rho \colon \U \to G$.

Given a surjective map $\P \colon \M \to \N$,  a $\P$-projectable action of $G$ on $\M$,  and the corresponding action of the quotient group $[G]$ on $\N$, we can thus introduce cross-sections for both actions, along with their associated moving frames and invariantization maps.  Assuming that the cross-sections are compatible, meaning that $\P$ maps one to the other, we deduce that the resulting invariantization maps are respected by the projection.

 \begin{proposition} Let $ \Kh$ be a local cross-section for the $\P$-projectable action of $G$ on $\widehat \U\subset  \M$, and $\K$ a local cross-section for the projected $[G]$-action on $\P(\widehat \U)=\U \subset \N$ satisfying the compatibility condition $\P(\widehat \K)=\K$. Let $\invh \, \colon \cFUh \to \cFUhG$ and $\inv\colon \cFU \to \cFU^{[G]}$ be the corresponding invariantization maps on smooth functions.  Then  
 \beq\label{lift_iota} \P\Upstar \inv(F)=\invh \>\P\Upstar (F) \qqbox{for all} F\in \cF(\U).\eeq 
\end{proposition}

If, furthermore, the actions of $G$ on $\Jk(\M,p)$ and $[G]$ on $\Jk(\N,p)$ are both locally free, then  the invariantization operation can be extended to differential forms in an analogous manner, as described in detail in \rf{KOivb}, and \eqf{lift_iota} readily generalizes from functions $F$ to differential forms $\Omega $.  


The construction of $\P$-related cross-sections is especially transparent in $\P$-canonical coordinates. 
As above, let $(\sups xp, \sups \v{n-p},\sups \y{m-n})=(x,\v,\y)$ be local coordinate  functions on $\M$, such that $\Pi(x,\v,\y)=(x,\v)$, with $x$ serving as independent variables on both $\M$ and $\N$, while $(\v,\y)$  and $\v$ serve as dependent variables on  $\M$ and $\N$, respectively.
Let $\K$ be a cross-section for the $[G]$-action on $\Jk (\N,p)$ and $\Kt=(\Pk)^{-1}(\K) \subset \Jk (\M,p)$. The cross-section  $\Kt$ can be prescribed by $m-\dim[G]$ independent algebraic equations involving only the variables $x, \v, \v_J^\alpha$.    
 There is a well-defined action of the global isotropy subgroup $G_N$ on $\Kt$.  
 Let $\Kh\subset\Kt$ be a cross-section for this reduced action. Since 
$G_N$ leaves the jet variables $x, \v, \v_J^\alpha$ fixed, the defining equations of $\Kh$ do not introduce any new relations among these variables, and thus $\Pk(\Kh) = \K$. By construction, $\Kh$ is a $G$-cross-section.

Assume now that there is a subgroup $\Gt \subset G$ that is isomorphic with the quotient group $[G]$.
In this case, $G$ factors as a product $G=G_N\cdot \Gt $,  and we can use  inductive construction developed in \cite{kogan-03} to determine the moving frame and the invariants.
(More generally, one can apply the general recursive algorithm in \cite{Orec} directly to the subgroup $G_N$ without requiring the existence of a suitable subgroup $\Gt$.) These constructions allow one to determine the formulae relating the invariants and invariant differential forms of the full group $G$ to those of the subgroups $G_N$ and, when it exists, $\Gt$. It turns  out that the preceding construction of  $\P$-related cross-sections interacts nicely with  the  inductive and recursive approaches, as described below.

 We note that the action of $\Gt $ on $\M$ projects to the $[G]$-action on $\N$ and $\Gt _N=\{e\}$. Let $\K\subset \Jk (N, p)$  be a cross-section for prolonged action of $[G]\cong\Gt $  and $\inv$ denote the corresponding invariantization map.   We observe that  $\Kt=(\Pk)^{-1}(\K) $ is a local cross-section  for the $\Gt $-action on $\Jk (M,p)$, and denote the corresponding invariantization map by $\invt$. Since the coordinates $(x,\v,\v^\alpha_J)$ are transformed by $\Gt $ in an identical manner, whether they are considered to be functions on $\Jk(N,p)$ or on $\Jk(M,p)$, we have  $\invt(x,\v,\v^\alpha_J)=\inv (x,\v,\v^\alpha_J)$. (By equality here, we mean that these functions have the same formulae, although they are defined on different spaces.)
Together with $ \invt(\y^\beta_J)$ they comprise a fundamental set of $\Gt$-invariants on $M$.

Assuming that  the order of prolongation  $k$ is at least the order of freeness of the $\Gt $-action on $\Jk (N, p)$, we  can invariantize the horizontal differential forms, $\invt(dx^i)=\inv(dx^i)$, where equality is again understood in the symbolic sense. We denote the horizontal parts of those forms by $\sups \omega p$ and the corresponding dual horizontal invariant differential operators by $\subs \CD p$. Since all of these objects are expressed in  terms of  $x,\v,\v^\alpha_J$ and $dx$ by the same formulae, whether they are defined  on $\Jk(\N,p)$ or $\Jk(\M,p)$, we will use the same symbols to denote them.

The action of $G_N$ restricts to the cross-section $\Kt$.  Let $\Kh\subset \Kt$ be a cross-section for this restricted action, and let  $\invh$ be the corresponding invariantization map.  Using the inductive method, we  can express the normalized invariants of $G$ in terms of the normalized invariants of $\Gt $ as follows:
\beq
\qxeq{
\invh (x^i) = \invt (x^i),\\
\invh (\v^\alpha_J)=\invt (\v^\alpha_J),\\
\invh (\y^\beta_J)= \invt \bbk{F^\beta_J\left (x,\,\v^\alpha_K,\,  \y^\gamma_K\right)},}
\eeq
where, $ \alpha$ runs from $1$ to $n-p$, while $\beta, \gamma$  run from $1$ to $m-n$, and  $J, K$ range over all multi-indices 
with $0 \leq \abs K \leq \abs J$. In the final formula, the $F^\beta_J$ are algorithmically computable functions.  We also note that invariantization $\invh$ preserves the $\widetilde G$-invariant basis of differential forms and differential operators:  $\invh(\omega^i)=\omega^i$ and $\invh(\CD_i)=\CD_i$.


\begin{example}  \label{ex-cp2}
 \rm Let us return to Example~\ref{ex-cp}, where we introduced canonical coordinates $(X,Y,Z)$ for the central projection, whose expressions in terms of the Cartesian coordinates are given by \eq{eq-c-can}.
   In this example, $G=\GL3$, $\Gt  =\SL3$,  $G_N=\R^*$, the latter denoting the one-dimensional Lie group of non-negative real numbers under multiplication, so that $[G]=G/G_N=\PGL3$.

 The standard cross-section  for  the projective action   \eq{eq-proj-l} of $[G]$  is  
 \beq\label{eq-cs-psl3-N}
 \K=\bigl\{ \xeq{X=Y=Y_1=0,\\ Y_2=1,\\Y_3=Y_4=0,\\Y_5=1,\\Y_6=0}\bigr\}\subset  \N,\eeq
 where $Y_i$ denotes the jet coordinate corresponding to $D_X^i(Y)$.
The lowest order normalized differential invariant is the standard projective planar curvature, $\inv(Y_7)=\,\prcurv $, whose explicit formula in jet coordinates  can be found in entry 2.3 of Table 5 in \cite E. 
The inductive method  \cite{kogan-03} enables one to express the projective curvature compactly in terms of the equi-affine curvature  as follows:
 \beq\label{prcurv} \prcurv=\frac{6\,\eacurv_{\eaal\eaal\eaal}\,\eacurv_{\eaal}-7\,\eacurv_{\eaal\eaal}^2-3\,\eacurv\,\eacurv_{\eaal}^2}{\,6 \,\eacurv_{\eaal}^{8/3}},\eeq
where the equi-affine curvature\footnote{In Blaschke \cite{Blaschke}, as well as in some other sources, the equi-affine curvature is defined to be $1/3$ of the expression $\eacurv$ in \eq{eacal}.  Our choice, however, leads to simpler numerical factors in the subsequent expressions.} and arc length are
 \Eq{eacal}
$$\qxeq{\eacurv = \frac \B{3\:Y_2^{8/3}}\,,\\d \eaal = Y_2^{1/3} \, dX,\quad \qbox{with}\quad  \B =  3\:Y_2Y_4-5\: Y_3^2.}$$
As usual, equi-affine invariants are  not defined at the inflection points $Y_2=0$.  Note also that $Y_2\equiv 0$ implies that the planar curve is (a part of) a straight line. The derivative of equi-affine curvature with respect to equi-affine arc length \eq{eacal} is given by
\beq\label{eq-ds-eacurv}\eacurv_\eaal = \frac A{9\: Y_2^4}\,,\eeq
where the differential function 
\Eq A
$$\A = 9\,Y_5\, Y_2^2-45\,Y_4\,Y_3\,Y_2+40\,Y_3^3$$
plays an important role in what follows.
   In particular, if $Y(X)$  satisfies $A \equiv 0$, then the equi-affine curvature of the curve is constant, and hence the curve must be contained in the orbit of a one-parameter  subgroup of the equi-affine group, which means that it is (part of) a conic section, \rf E. 
Otherwise, the projective arc length element and dual invariant differential operator are given by
\Eq{pral}
$$\qqeq{d\pral= \left(\, {\eacurv_\eaal} \right)^{1/3} d \eaal = \frac{A^{1/3}}{3^{2/3} \,Y_2}\,dX,\\
\CD_\pral= \frac{3^{2/3} \,Y_2}{A^{1/3}}\,D_X,}$$
Planar projective invariants are defined at the points where $Y_2\neq 0$ and $A\neq 0$, and are generated by the projective curvature invariant $\eta$ through invariant differentiation with respect to the projective arc length \eq{pral}.

We now employ the cross-section   $\Kt=(\Pk)^{-1}(\K) \subset \M$, defined by the same set of equations \eq{eq-cs-psl3-N} as $\K$, to compute differential invariants for the $\Gt=SL(3)$-action on $M$.
As above, the gauge invariants are generated by the invariant $\prcurvh =\invt(Y_7)$ and the invariant differential operator $\CD_{\pbal}=\invt (D_X)$, which, in the canonical coordinates, have the same symbolic expressions  as their planar counterparts $\prcurv$ and  $\CD_\pral$. Nonetheless,  we will be using hats to emphasize that the former are defined  on $M$, and to be consistent  with the notation of \sect{sect-inv-diff-forms}.

The computation of the invariantizations $\invt(Z_i)$, $i \geq 0$, of the fiber coordinates $Z_i=D_X^i\,Z$ requires more effort. 
 We note that the prolongation of   \eq{G-canonical} is given by 
\beq\label{eq-SLz}
\qxeq{Z \longmapsto  \bar Z = (a_{31}\,X+a_{32}\,Y+a_{33})\,Z ,\\ Z_{i+1} \longmapsto H\, D_X(\bar Z_i),\\ i\geq 0,}
\eeq
where 
$$H =\frac{(a_{31}\,X+a_{32}\,Y+a_{33})^2}{(a_{31}a_{12}-a_{11}a_{32})\,XY_1+(a_{12}a_{33}-a_{13}a_{32})\,Y_1+(a_{32}a_{11}-a_{12}a_{31})\,Y+(a_{33}a_{11}-a_{13}a_{31})}.$$ 
The moving frames $\rho\colon\Jk(\N,1)\to[G]$  and  $\widetilde\rho\colon \Jk(\M,1)\to \Gt$ corresponding to the respective cross-sections $\Kt$ and $\K$  have the same symbolic expressions in the canonical coordinates. Since the explicit formulas are rather involved, we will not reproduce them here, but     refer the reader to Example~5.3 in  \cite{kogan-03}, where the projective moving frame is expressed in a concise way using the inductive approach.      The normalized invariants $\ostrut{14}0\itZ = \invt(Z)$ and $\itZ_i = \invt (Z_i)$, $i \geq 1$, are obtained by substituting those expressions into 
\eq{eq-SLz}.
In particular,  
\beq\label{eq-invtz}\itZ = \invt(Z)=\frac{Z}{\eacurv_\eaal^{1/3}},\eeq
where $\eacurv_\eaal$, given by \eq{eq-ds-eacurv}, is now considered to be a function on $\J5(\M,1)$.

We conclude that a complete system of centro-equi-affine invariants for space curves is generated by the seventh order gauge invariant $\prcurvh$, whose symbolic formula is \eq{prcurv} and the fifth order differential invariant $\itZ\ostrut{14}0$ in \eq{eq-invtz}, by successively applying the invariant differential  operator $\CD_\pbal$, whose symbolic formula is given by \eq{pral}. Remarkably, $\prcurvh$ is the projective curvature and $\itZ$ is $z$ times an equi-affine invariant of the image curve. In \sect{sect-P0}, we will express $\prcurvh$ and $\ostrut{14}0\itZ$ in terms of the third and fourth order centro-equi-affine invariants derived in \cite{Ocentro}.

Finally to compute the centro-affine differential invariants, for $G=\GL3$, we consider the action of $G_N \simeq \R$ on $\M$ given by 
$$X\longmapsto   X, \qquad Y\longmapsto  Y, \qquad Z\longmapsto  \lambda \:Z.$$
This has a simple prolongation:
$$ Y_i \longmapsto Y_i, \qquad Z_i \longmapsto \lambda \:Z_i, \qquad i>0.$$ 
The $G_N$-action restricts to $\Kt$, and we define a cross-section $\Kh\subset \Kt$ to the restricted action by appending the  equation $Z=1$ to \eq{eq-cs-psl3-N}. Following the inductive approach, we observe that $\Kh$ is a cross-section for the prolonged action on  $\J6(\M,1)$ and that the normalized $G$-invariants are expressed in terms of the normalized $\Gt =\SL3$ invariants as follows,
\begin{eqnarray}  \label{eq-ivh}
\qeq{\invh(Y_i) =\invt(Y_i), \\i>6,\\ \\
\ihZ_i=\invh(Z_k) = \frac {\invt (Z_k)}{\invt (Z)},\\ k>0,}
\end{eqnarray}
where we omit the constant phantom invariants. 

Thus, the centro-affine differential invariants for space curves are  generated by the same seventh order gauge differential invariant
$\prcurvh$ and the sixth order differential invariant 
\beq\label{eq-invtz1}\ihZ_1=\invh(Z_1)=\frac{\invt (Z_1)}{\invt (Z)}= \frac{Z_1}Z\,\frac 1 {(Y_2\,\mu_\eaal)^{1/3}}\ -\ \frac{\mu_{\eaal\eaal}}{3\,\mu_\eaal^{4/3}}\eeq
through successive application of the invariant differentiation operator 
\beq\label{eq-invD1}\CD_\pbal = \invh(D_X)=\invt(D_X).
\eeq

\end{example}

\Rmk
It may be instructive to revisit the preceding example in the standard jet coordinates: $(x, y, z, y_1, z_1,\dots)$, where $ y_i=D_x^iy$ and $z_i=D_x^iz$.
The corresponding non-coordinate cross-section $\Kt\ostrut{14}0$ is given by 
 \Eq{eq-cs-psl3-1}
$$\ceq{\qxeq{x=0, & y=0,&  y_1 = 0, & y_2=1,&y_3 = -3\,z_1,} \\
\qxeq{ y_4 = 12\,z_1^2-6\,z_2,&
y_5 = -60\,z_1^3+60\,z_1\,z_2-10\,z_3+1,}\\
\qeq{y_6 = 360\,z_1^4-540\,z_1^2\,z_2+120\,z_1\,z_3+90\,z_2^2-24\,z_1-15\,z_4.}}
$$
The cross-section $\Kh$ is fixed by appending the further equation $z=1$ to \eq{eq-cs-psl3-1}.
   We note that 
   \Eq{eq-pral-alt}
$$d\widehat\pral=\invh \left({\Po\pss 5}\,dX\right)\equiv\invh\left(\frac {z-z_1x}{z^2}\,dx\right)=\invh(dx),$$
where, in the middle term, $dX$ is considered to be a form on $N$ and, as usual, $\equiv$ means equality up to a contact form.
The invariant form $d\widehat\pral$ is dual to the invariant differential operator \eqref{eq-invD1}.
Applying the moving frame recurrence formulae and the Replacement Theorem \eqref{Replacement}, we can express the projective curvature $\prcurvh= \invh ( Y_7)$ in terms of normalized invariants $I_i=\invh(y_i)$,  
$J_i=\invh(z_i)$, as follows:
\Eq{eq-mh-alt} 
$$\eeqo{\prcurvh
= I_7+ 27\,\bigl(120\,J_1^5-240\,J_1^3\,J_2+60\,J_1^2\,J_3 + 90\,J_1\,J_2^2\cr& \hskip150pt{}
-20\,J_1^2-10\,J_1\,J_4-20\,J_2\,J_3+4\,J_2+J_5\bigr)\\
= 3^{1/3} ( - \:D_\pbal J_1 + J_1^2 + J_2) .}
$$

\Section3{Non  projectable actions and some applications}

We now turn our attention to the important case, arising in image processing and computer graphics, of central and parallel projections of three-dimensional space curves to the two-dimensional plane. Central projections model pinhole cameras, while parallel projections provide  a good approximation for a pinhole camera when the distance between a camera and
an object is significantly greater than the object depth, \cite{HZ}. 
The formulation of \sect2 does not entirely cover these examples, since the associated group action of the affine group on $\R^3$ is not projectable. To handle such cases, in general, we  identify a  subgroup $H$ of the entire group $G$ that acts projectably with respect to a surjective map $\Pi_0$. Usually $H$ is chosen to be the maximal such subgroup. We then construct a family of surjective maps $\Pi_g\colon M\to N$ parameterized by  elements of $G$ and examine the relationship between differential $H$-invariants of submanifolds of $M$ and invariants of the  family of  projections of these submanifolds. In Section~\ref{families}, we describe this relationship in the general setting of abstract manifolds and group actions. In Sections~\ref{sect-P0} and~\ref{sect-Pc}, we specialize to the concrete case of the central projections of planar curves, while Sections~\ref{sect-P0-p} and~\ref{sect-Pc-p} treat the case of parallel projections.

\subsection{Non-projectable actions and induced families of maps} 
\label{families}

We start, as above, with a fixed surjective map $\Po\colon \M \to \N $, but now suppose that the group $G$ acts non-projectably on $\M $.  Assume further that there exists a (nontrivial) subgroup $H \subset G$ whose action is $\Po$-projectable.  In this situation, we  define a family of surjective maps and corresponding projectable subgroup actions.  

Recall, first, the \emph{adjoint} or conjugation action of a group on itself, denoted by
\Eq{Ad}
$$\Ad g\, (h) =  g\,h\, g^{-1} \qbox{for }g,h \in G.$$

\begin{theorem}\label{th-family}
Let $\Po\colon \M \to \N $  be a  surjective map.  Suppose that $G$ acts on $\M $ and, moreover,  $H\subset G$ is a proper subgroup whose action on $\M $ is $\Po$-projectable.  For each $ g\in G$,  define the $g$-transformed surjective map $\P _g=\Po\comp g^{-1}\colon \M \to \N $. Then the action of the conjugate subgroup 
$H_g = \Ad g\, (H) = g H g^{-1} \subset G$ 
is $\P _g$-projectable.
\end{theorem}

\begin{proof} 
Assume that $\z,\zt \in \M $ belong to the same fiber of $\P _g$, namely:
\beq\label{eq-zz'} 
\P _g(\z) = \Po(g^{-1}\cdot \z) = \Po(g^{-1}\cdot \zt\,) = \P _g(\zt\,) .\eeq
Since the action of $H$ is $\Po$ projectable, \eq{eq-zz'} implies
$$\Po(h\:g^{-1}\cdot \z)=  \Po(h\:g^{-1}\cdot \zt\,) \ \qbox{for all} \
h\in H.$$
  Inserting  the identity element in the form  $g^{-1} g$ in the above equality, we obtain
$$\P _g( g\:h\:g^{-1}\cdot \z) =\Po(g^{-1} g\:h\:g^{-1}\cdot \z) =  \Po(g^{-1} g\:h\:g^{-1}\cdot \zt \,) =  \P _g (g\:h\:g^{-1}\cdot \zt \,),$$
which implies that the action of $H_g$ is $\P _g$-projectable.
\end{proof}

\Rmk 
If $H_\N $ is the global isotropy group of the $\Po$-projection of the action of $H$ on $\N $, then $H_{\N,g} = \Ad g\, (H_\N) $ is the global isotropy group of  the $\P _g$-projection of the action of $H_g$ on $N$. (Keep in mind that, while $H_\N $ is a normal subgroup of $H$,  it need not be a normal subgroup of $G$.)  Setting $[H_g] =H_g/H_{\N,g}$, we can therefore express the $[H_g]$-differential invariants of the images of submanifolds under $\P _g$ in terms of the $H_g$-differential invariants on $\M $.


We finally state a simple, but  useful relation between the pull-backs of functions under $\Po$ and $\P_g$:
\beq\label{eq-PoPg}\P_g\Upstar \,F(z)=(\Po \comp g^{-1})\Upstar F(z)=\Po\Upstar \,F(g^{-1}\cdot z),
\eeq
for any $F\colon\N\to\R$ and $z\in \M$.
 \begin{example} \rm Let $\M = \R^3$ and $\N = \R^2$. Consider the standard orthogonal  projection 
$\Po(x,y,z)=(x,y)$.  Let $G=\R \ltimes \R^3$ be a four-dimensional semi-direct product group, parametrized by $a,b,c,d$, that acts on $\M $ via the transformations
$$g\cdot (x,y,z)=(x+a\:z+b,\ y+c,\ z+d).$$
Although this action is not $\Po$-projectable, the translation subgroup  $H =\{(0,b,c,d)\}\cong \R^3 \subset G$, 
does act $\Po$-projectably.
The global isotropy subgroup $H_N$ is isomorphic to $\R$, and its fiber-preserving action is given by   $(x,y,z)\mapsto (x, y,z+d)$.  If $(X,Y)$ denote coordinates on $\N$, then the quotient group $[H]=H/H_N$ acts on $\N$ by translation: $(X,Y)\mapsto (X+b,Y+c)$.

In accordance with our general construction, we define the family of surjective maps $\P _g\colon \M\to \N$  by 
$$\P _g(x,y,z)=\Po (g^{-1} \cdot (x,y,z)) = (x-a\:z-b,y-c).$$
Since $H$ is a normal subgroup of $G$, its conjugate subgroups coincide, $H_g=H$, and thus all the surjective maps $\P_g$ are $H$-projectable. Moreover $H_{N,g}=gH_Ng^{-1}=H_N$, but its fiber-preserving action $ (x,y,z) \mapsto (x+a\:d, y,z+d)$ depends on $g$, or, rather, on the first parameter $a$ of $g$, since it parametrizes the cosets $gH$.   The $\P_g$ projection of the $[H]$-action to $\N$ is given by
$$(X,Y) \longmapsto (X+b-a\:d, Y+c) \qbox{for} (b,c,d) \in H.$$  
Observe that this family of $[H]$-actions are all translations, but parametrized by the value of $a$.

We assume, for simplicity, our space curves are given as graphs $\C=\bc{(x,y(x),z(x))}$.  Under the action of the translation subgroup $H$, the invariant differential operator is $\CD = D_x$, and the two first order differential invariants $y_x,z_x$ comprise a generating set for the entire differential invariant algebra. On the other hand, for a plane curve parametrized by $\left(X,Y(X)\right)$, the single differential invariant  $Y_X$ forms a generating set.
The map $\P _g$ projects the space curve $\left(x,y(x),z(x)\right)$ to the plane curve
\beq\label{eq-Pggamma}\P _g(\C)=\left(X(x),Y(x)\right)=\left(x-a\,z(x)-b, y(x)-c\right).\eeq
Moreover,
$$ (\P_g^{(1)})\Upstar Y_X =\frac{y_x }{1-a\,z_x}$$
provides the  relationship between the generating differential invariants of the space curve and its planar image evaluated at corresponding points. It can be obtained either by computing the first prolongation of \eq{eq-Pggamma}, or by using our general formula \eq{eq-PoPg}, which in this case amounts to  $\P_g^{(1)}\:\Upstar Y_X=(g^{-1})^{(1)}\cdot y_x$.  The appearance of the parameter $a$ is due to the non-projectability of the full action. 
We also note that the invariant one-form $\omega = dX$ on $\N$ is pulled back via $\P _g$ to the $H_g$-contact-invariant horizontal differential form
$$\P _g\Upstar  \omega \equiv \lomega =  (1-a\,z_x)\,dx ,$$
again depending upon the parameter $a$ that determines the conjugacy class of $g$.
Theorem \ref{th-DI} then enables us to determine relations between the higher order differential invariants by applying the dual total invariant differential operator 
$$\CE = \frac{d}{dX}=\frac 1 {1-a\,z_x}\>\frac{d}{dx},$$
in accordance with formula \eq{CDI}.  

\end{example}

When the subgroup $H \subset G$ is not normal, the following proposition relating moving frames and  invariants under the adjoint action of $G$ on $H$ will be useful.

\begin{proposition}\label{prop-conj-inv} 
Let $G$ act on $\M$, and let $H\subset G$ be a subgroup. Given a fixed element $g\in G$, let $H_g = \Ad g\, (H) = g\, H\,g^{-1}$ denote the conjugate subgroup.
\begin{enumerate}
\item If $I\colon \M\to\R$ is an $H$-invariant function then $ I_g =I\comp g^{-1}$ is an $H_g$-invariant function. 
\item If $\rho\colon \M\to H$ is the moving frame for the $H$-action corresponding to the cross-section $\K \subset \M$, then  $\rho_g\colon \M\to H_g$ defined by  
$$\rho_g(z)=\Ad g\, \comp \rho(g^{-1}z) = g\cdot \rho(g^{-1}z) \cdot g^{-1}$$ 
is the moving frame for the $H_g$-action corresponding to the transformed cross-section  $\K_g =g\cdot \K $.
\item If $\inv(F)(z) =F(\rho(z) \cdot z)$ is the $H$-invariantization, corresponding to the cross-section $\K$,  of the function $F\colon \M\to\R$ then   
\beq\label{tildeiota}
\inv_g(F)(z)=F\left(\rho_g(z)\cdot z\right)= F\left(g\cdot \rho(g^{-1}z)\cdot  g^{-1} z\right ) =\inv(F\comp g)(g^{-1}z)
\eeq 
is the invariantization of $F$ for the $H_g$-action, corresponding to the cross-section $\K_g$. 
\end{enumerate}
\end{proposition}

\emph{Warning\/}: Since equation \eq{tildeiota} can be summarized by  
$$\inv_g(F)=\inv(F\comp g)\comp g^{-1} \quad  \text{ or, equivalently,  } \quad  \inv_g(F\comp g^{-1})=\inv(F)\comp g^{-1},$$
  it is important to underscore that  $\inv(F\comp g)\neq \inv(F)\comp g$. Indeed, 
$$\inv(F\comp g)(z)=F(g\cdot \rho(z)\cdot z), \qbox{while} \inv(F)\comp g(z)=F(\rho(g\cdot z)\cdot (g\cdot z)).$$

\subsection{Central projections from the origin}
\label{sect-P0}

In this and the following section, we specialize the preceding results to the case of central projections.  We begin by assuming the center of the projection is at the origin.
Let $\rr = (x,y,z)$ be  the standard coordinates on $\M  = \R^3$ and $\w = (X,Y)$ be the standard coordinates on $\N  = \R^2$.
Let $\Jk(\M ,1)$ denote the $k$-th order jet space associated with space curves $\C \subset \M$.  Treating $x$ as the independent variable, the corresponding jet coordinates are denoted by  $\rr^{(k)} = (x, y, z, y_1,z_1,\dots, y_k,z_k)$, where $y_i, z_i$ correspond to the $i$-th order derivatives of $y,z$, respectively, with respect to $x$.   Similarly, let $\Jk(\N ,1)$ denote the $k$-th jet space of plane curves, with coordinates $\w^{ (k)} = (X,Y, Y_1,\dots, Y_k)$, where $Y_i$ corresponds to the $i$-th order derivative of $Y$  with respect to $X$.

 Let us first consider the case of central projection, centered at the origin, from  $\M = \{ (x,y,z)\,|\,z \neq 0\} \subset \R^3$
 to the plane $\N \simeq \R^2$ defined by $z=1$. We will work in the coordinate system on the image plane provided by the first two coordinate functions on $\M $, i.e. $X(x,y,1)=x$ and  $Y(x,y,1)=y$. 
  As in \eq{eq-P0Z}, the central projection map $\Po\colon\M \to\N=\R^2 $ is thus explicitly given by
 \beq\label{P0}
 (X,Y)=\Po(x,y,z)=\left(\frac {x}{ z}, \ \frac {y}{ z}\right).
 \eeq
 As we noted in Example~\ref{ex-cp}, the linear action of $G=\GL3$ on $\M$ is  $\Po$-projectable and induces the projective action of $[G]=\PGL3$ on $\N \subset \RP^2$ given by \eq{sect-inv-diff-forms}.

\Rmk The centro-affine action of the linear group on $\M $ is $\Po$-projectable, because linear maps take central projection fibers to fibers. On the other hand, translations do not respect the fibers, and hence, the action of the translation subgroup $\R^3$, as well as the action of the full affine group $\AG3$, is \emph{not}  $\Po$-projectable, and does not project to a well-defined action on $\N $.    The quotient $\AG3/\GL3$ parametrizes the family of central projections considered in \sect{sect-Pc}. 


Our goal is to relate the projective differential invariants of the projected curve to the centro-affine differential invariants of the originating space curve.
Let $A\ps k$ denote the prolongation of the linear map induced by $A \in \GL3$ to the $k$-th jet space $\Jk(\M ,1)$. Similarly, the prolonged action of $[A] \in \PGL3$
on $\Jk(\N ,1)$ will be denoted by $[A]\ps k$. 
Applying the transversality condition \eq{DPmax} to \eq{P0}, we conclude that a jet $\yk \in \Jtok(\M,1)$ is $\Po$-regular if and only if the total derivative matrix
 \beq\label{cproj-reg} 
D\P_0 = \left(D_x\left(\frac {x}{ z}\right), \ D_x\left(\frac {y}{ z}\right)\right)  = \left(\frac {z - x\,z_x}{ z^2}, \ \frac {y_x \,z - y\, z_x}{ z^2}\right)
 \eeq
has rank $1$, which requires that the two numerators, $z - x\,z_x,\ y_x \,z - y\, z_x$, cannot simultaneously vanish.  Geometrically, this implies that the curve intersects the fibers, i.e.~the rays through the origin, transversally.

Let $\C \subset \M$ be a smooth space curve parametrized by $(x,y(x),z(x))$.  Its projection  $\Cp =\Po(\C)\subset\N $ has induced parametrization
 \beq\label{eq-p0} \left(X(x),Y(x)\right)=\left(\frac {x}{ z(x)}, \ \frac {y(x)}{ z(x)}\right).\eeq
The explicit formulae for the   $k$-th prolongation $\Pk_0\colon \Jtok(\M,1)\to \Jk(\N,1)$  are  given inductively by
 $$
 \qeq{X=\frac {x}{ z},\\ Y=\frac {y}{ z},\\ Y_1=\frac{D_x Y}{D_x X}=\frac{y_x \,z - y\, z_x}{z - x\,z_x},\\  Y_i=\frac{D_x Y_{i-1}}{D_x X} =\frac{z^2\,D_x Y_{i-1}}{z - x\,z_x},\\ i>1, }$$ 
on the open subset of  $\Jtok(\M,1)$ where  $z - x\,z_x\neq 0$.  Geometrically, the latter inequality requires that the space curve not be tangent to any plane of the form $z=c\,x$ for $c$  constant, and hence its projection not have a vertical tangent at the corresponding point.


Theorem \ref{th-dinv} immediately implies: 
 \begin{theorem} \label{th-inv0} If $I\colon \Jk(\N,1)\to\R$ is a differential invariant for the projective action of $\PGL3$ on $\N $, then $\Ih = I\comp \Pok\colon \Jk(\M,1)\to\R$ is a differential invariant for the centro-affine action of $\GL3$ on $\M $.
 \end{theorem}
\Rmk Theorem~\ref{th-inv0} remains valid if we replace $\Po$ with any projection centered \emph{at the origin} to an arbitrary plane, because the images  of a space curve under projections with the \emph{same} center are all related by projective transformations.

We now seek to express the projective curvature $\prcurv$ of the projected curve $\Po(\C)$ in terms of  centro-affine  differential invariants of $\C$. We begin by summarizing the equivariant moving frame calculations in \cite{Ocentro}.  We choose the cross-section to the prolonged centro-affine action on $\Jk(M,1)$ defined by the normalization equations 
\beq\label{eq-cs}
\weq{x = 0,\\y = 0,\\z = 1,\\y_1=0,\\z_1 = 0,\\y_2 = 1,\\z_2 = 0,\\y_3 = 0,\\y_4 = 3.}
\eeq
(The reason for this non-minimal choice of  the cross-section will be explained below.)  Replacing the jet coordinates $\rr ^{(k)}$ by their transformed versions $\widetilde {\rr ^{(k)}} = A\ps k \cdot \rr^{(k)}$ for $A\in\GL3$, and solving the resulting equations for the group parameters produces the moving frame $\rho \colon \J3(M,1) \to \GL3$.  The resulting normalized differential invariants are then obtained by invariantization of the higher order jet coordinates:
\Eq{IkJk}
$$\qxeq{I_k = \inv(y_k),\\J_k = \inv(z_k).}$$
The invariantization of the lower order jet coordinates used to define the cross-section produces the phantom invariants whose values coincide with moving frame normalization constants in \eq{eq-cs}:
\beq\label{eq-phdi}
\sxeq{I_0 = \inv(y) = 0,&J_0 = \inv(z) = 1,&I_1 = \inv(y_1)=0,&J_1 = \inv(z_1) = 0,\\I_2 = \inv(y_2) = 1,&J_2 = \inv(z_2) = 0,&I_3 = \inv(y_3) = 0,&I_4 = \inv(y_4) = 3.}
\eeq
The remaining normalized invariants, i.e.~$I_k$ for $k \geq 5$ and $J_l$ for $l \geq 3$ form a complete system of functionally independent differential invariants for the centro-affine action.

To write out the explicit formulas, as found in \cite{Ocentro}, 
we use 
$$\bp {\rr_1}{\rr_2}{\rr_3} = \tp {\rr_1}{\rr_2}{\rr_3}$$
 to denote the determinant of the $3\times 3$ matrix with the indicated (row) vectors, or, equivalently, their vector triple product. 
Suppose the space curve is parametrized by $\rr(t) = (x(t),y(t),z(t))$.
Let
\Eq{dscea}
$$ds = \Delta ^{1/3} \,dt,\qbox{where} \Delta = \bp \rr{\rr_t}{\rr_{tt}} $$
denote the \emph{centro-equi-affine arc length element} with corresponding invariant differentiation 
\Eq{Dscea}
$$ D_s = \frac1 {\Delta^{1/3}}\, D_t.$$
Thus, when parametrized in terms of arc length, the curve satisfies the unimodularity constraint 
\Eq{unimodula}
$$[\rr,\rr_s,\rr_{ss}]=1.$$

\Rmk We exclude singularities where $\Delta = 0$.  A space curve is \emph{totally degenerate} when $\Delta \equiv 0$ at all points; this is equivalent to the curve $\C\subset P_0$ being contained in the plane $P_0 = \ro{span}\{\rr(0),\rr_t(0)\}$ spanned by its initial position and velocity.

The \emph{centro-equi-affine curvature} and \emph{torsion} differential invariants are given by
\beq\label{ktcea}
\eeqo{\kappa = 
-\,\frac 12 D_t^2 \Pa{\frac1{\Delta ^{2/3}}} + \frac{\bp \rr{\rr_{tt}}{\rr_{ttt}}}{\Delta ^{5/3}} = \bp 
\rr{\rr_{ss}}{\rr_{sss}},\\ 
\tau =  \frac{\bp {\rr_t}{\rr_{tt}}{\rr_{ttt}}}{\Delta ^{2}} = \bp {\rr_s}{\rr_{ss}}{\rr_{sss}}.}
\eeq
 Note that $\kappa $ is a fourth order differential invariant\footnote{The second expression in these formulas is potentially misleading; keep in mind that the arc-length element \eq{dscea} involves second order derivatives of the curve's parametrization.}  while $\tau $ is a third order differential invariant. (This is in contrast to Euclidean curves, where torsion is the higher order differential invariant.) 
 
\emph{Warning}: We have switched the designation of $\kappa $ and $\tau $ from that used in \cite{Ocentro}, and also deleted a factor of $3$ in $\tau$ to slightly simplify the formulas. Our choice of notation is motivated  by the fact that the condition $\tau = 0$ is equivalent to the curve being contained in a plane $P \subset \R^3$, thus mimicking the Euclidean torsion of a space curve.

\bigskip

As in \cite{Ocentro}, differentiating \eq{unimodula} produces $[\rr,\rr_s,\rr_{sss}]=0$, which, when compared with \eq{ktcea}, produces the associated \emph{Frenet equation}
 \beq\label{eq-mv} \rr_{sss}=\tau\: \rr-\kappa \:\rr_s.\eeq
Consequently, the condition $\tau=0$ is equivalent to $\rr_{sss}$  and $\rr_s$ being collinear, while  $\kappa=0$ is equivalent to the collinearity of $\rr_{sss}$  and $\rr$.

Under uniform scaling $\rr \longmapsto \lambda \: \rr$ the centro-equi-affine differential invariants and arc length scale according to 
$$\qxeq{\kappa \longmapstox \lambda ^{-2} \:\kappa ,\\\tau \longmapstox \lambda ^{-3} \:\tau ,\\ds \longmapstox \lambda \, ds.}$$
Assuming that\footnote{If $\kappa < 0$ just replace $\sqrt \kappa \,$ by $\sqrt{-\, \kappa }\, $ in the formulas.} $\kappa > 0$, we can therefore take
\beq\label{kt-ca}
\qxeq{\kh = \frac{\kappa _ s}{\kappa ^{3/2}},\\
\th = \frac{\tau}{\kappa ^{3/2}}\,,}
\eeq
as the fundamental centro-affine differential invariants, with orders $5$ and $4$, respectively.
Similarly, the \emph{centro-affine arc length element} is
\beq\label{ds-ca}
d \caal = \sqrt\kappa \> ds = \inv (dx),
\eeq
with dual invariant derivative operator
\beq
D_\caal = \kappa^{-1/2} D_s = \frac1 {\Delta^{1/3}\,\sqrt\kappa} \, D_t.
\eeq

\Rmk There is a second independent fourth order differential invariant, namely
\beq\label{bh}
\bh = \frac{\tau_s}{\kappa^2} = \th_ \sigma + \fr32\,\kh\,\th.
\eeq
Note that both terms of the right hand side of this formula are of order $5$, and hence the terms involving fifth order derivatives cancel.   One could, alternatively, use $\th,\bh$ as generating differential invariants, although the resulting formulas become more complicated.  A similar observation applies to the pair of fourth order generating differential invariants 
\beq\label{kttt}
\qxeq{\widetilde \kappa = \frac{\kappa}{\tau ^{2/3}},\\\widetilde \tau = \frac{\tau_s}{\tau ^{4/3}},}
\eeq
which results from the minimal moving frame cross-section
$$\weq{x = 0,\\y = 0,\\z = 1,\\y_1=0,\\z_1 = 0,\\y_2 = 1,\\z_2 = 0,\\y_3 = 0,\\z_3 = 3.}$$
The fact that the generating invariants \eq{kt-ca} lead to the simplest formulae for the projective curvature of the projected space curve is one of the key reasons for our choice of non-minimal cross-section \eqref{eq-cs}.

The Replacement Theorem \eqref{Replacement} implies that if $\Ih \colon \Jk(\M,1) \to \R$ is any centro-affine differential invariant, then its explicit formula in terms of the normalized centro-affine invariants can be obtained by invariantizing  each of its arguments:
$$\ibeq{200}{\Ih(x,y,z,y_1,z_1,y_2,z_2,y_3,z_3,y_4,z_4,y_5,z_5,\ldots y_n,z_n) \\= I(0,0,1,0,0,1,0,0,J_3,3,J_4,I_5,J_5,\ldots ,I_n,J_n).}$$
Applying this result to the pull-back $\prcurvh = \prcurv \comp \Po\ps 7$ of the projective curvature invariant \eq{prcurv} produces the desired formula 
\beq\label{mh-proj1}
\prcurvh= \frac {3\:(I_5+10\:J_3)(2\:I_7+42\:J_5 - 105\:(I_5 + 4\:J_3)) - 7\: (I_6+15 \:J_4-45)^2}{6\,(I_5+10\:J_3)^{8/3}}
\eeq
that expresses the projective curvature of the central projection of a nondegenerate space curve in terms of its normalized centro-affine differential invariants \eq{IkJk}.  Alternatively, the moving frame recursion formulas, \cite{FOmcII,Ocentro}, can be employed to express the higher order normalized differential invariants in terms of invariant centro-affine derivatives of $\kh,\th$. Applying the general algorithm, we find
\Eq{carf}
$$\eeqo{
J_3 = \th,\\
J_4 = D_\caal J_3 + \f2\:I_5J_3 + 2\:J_3^2 
 = \th_\caal + \fr32\:\kh\:\th,\\
J_5 = D_\caal J_4 + \fr23\:I_5J_4 + \fr83\:J_3\:J_4 + 9\:J_3 
 = \th_{\caal\caal} + \fr32\:\kh_\caal\:\th + \fr72\:\kh\:\th_\caal + 3\:\kh^2\th + 9\:\th,\\
I_5 = 3\:\kh - 4\:\th,\\
I_6 = D_\caal I_5 + \f2\:I_5^2 + 2\:I_5J_3 -5\:J_4 + 45 
 = 3\:\kh_\caal - 9\:\th_\caal + \fr92\:\kh^2 - \fr{27}2\:\kh\:\th + 45,\\
I_7 = D_\caal I_6 + \fr23\:I_5 I_6 + \fr83\:I_6 J_3 + 21\:I_5 - 6\: J_5 - 60\:J_3  \\
= 3\:\kh_{\caal\caal} - 15\:\th_{\caal\caal} + 15\:\kh\:\kh_\caal - \fr{45}2\:\kh_\caal\:\th - \fr{105}2\:\kh\:\th_\caal + 9\:\kh^3 - 45\:\kh^2\th + 153\:\kh - 198\:\th,
}
$$
and so on.  One can, of course, easily invert these formulae to write  $\kh,\th$ and their derivatives in terms of the normalized differential invariants $I_k,J_k$. We note that $I_3$ and $J_5$ generate the differential algebra of invariants through the differential operator $D_\caal$.
 
The resulting formula for $\prcurvh$ has a particularly simple form if we set
\beq\label{ah}
\qxeq{ \ah = \kh + 2\:\th = \frac{\kappa _ s + 2\:\tau}{\kappa ^{3/2}}=\fr 1 3\, (I_5+10\,J_3)\,.}
\eeq
Namely,
{\beq\label{mh-proj}
\eeqo{\prcurvh = {3^{-2/3}}\,\left(\ah^{-5/3}\: \ah_{\caal\caal} - \fr76 \:\ah^{-8/3}\: \ah_{\caal}^2 + \fr32\:\ah^{-2/3}\: (\kh_{\caal} + \f4\: \kh^2 - 1)\right)
\\= {3^{1/3}}\,\left(-2\:\ah^{-1/2} D_\caal^2(\ah^{-1/6}) + \fr1 2\:\ah^{-2/3}\: (\kh_{\caal} + \f4\: \kh^2 - 1)\right) .}
\eeq}

As we discussed in Example~\ref{ex-cp2}, projective curvature is undefined for straight lines --- equivalently $Y_2\equiv 0$ --- and conics --- equivalently $A\equiv0$, where $A$ is given by \eq{A}. We have
\beq\label{eq-line-conic}{\P\pss5_0}(Y_2)=-\frac{z^3\,\Delta}{(x\,z_1-z)^3}\, \qbox { and } {\P\pss5_0}(\A)=\frac{27\, z^{15}\,\Delta^4}{(x\,z_1-z)^{12}}\, \,(\kappa_s+2\:\tau)\eeq
The first condition tells us that a space curve is projected to a line segment if and only if $ \Delta \equiv 0$ and hence, as we noted earlier, it lies on the plane passing through the origin. The second condition tells us that a curve  projects to a conic if and only if $\Delta\neq 0$ and $\kappa_s+2\tau \equiv 0$, which, assuming $\kappa \ne 0$, is equivalent to the vanishing of the differential invariant $\ah\equiv 0$.

Recall, \cite E, that, in general, a nondegenerate curve has all constant differential invariants if and only if it is (part of) the orbit of a one-parameter subgroup. For example, the twisted cubic $\C$, parametrized by $(x,x^2,x^3)$, has constant centro-affine curvature and torsion invariants $\kh=-{4}/{\sqrt 3}$, \ $\th ={2}/{\sqrt 3}$, and can be identified as an orbit of the one-parameter subgroup of diagonal matrices $\set{\rm{diag}(\lambda,\lambda^2,\lambda^3)}{\lambda\neq 0}$. Further, we note that the differential invariant \eq{ah} vanishes, $\ah=\kh+2\:\th=0$ on $\C$, reflecting the fact that the twisted cubic is projected to a parabola under $\Po$.

\Rmk\hskip-2mm\footnote{This remark is significantly changed in comparison with the version of the paper published in  Lobachevskii J.~Math.~{\bf 36} (2015), 260--285.  Several formulae are corrected and new formulae are inserted. To preserve the numbering in subsequent sections, we  added * to the additional formula tags in the remainder of this section.}  In Example~\ref{ex-cp2}, we introduced another invariant differential form, the pull-back of the projective arc length element \eq{pral}. We find that
\beq\label{pull-dsigma}d\pbal = \Po\pss 5 \,d\pral \equiv (3\,\ah)^{1/3 } \,d\caal = (3\,\alpha)^{1/3}\,ds,\eeq
where, as before, $\ah=\kh+2\:\th$, and we set 
$$\alpha=\ah\,\kappa^{3/2}=\kappa_s+2\,\tau,$$
 while $d\caal$  and  $ds$ are  given by  \eq{ds-ca} and \eq{dscea}, respectively. 
 
We  showed in Example~\ref{ex-cp2} that $\prcurvh$ and $\ihZ_1=\invh(Z_1)$, given by \eq{prcurv} and  \eq{eq-invtz1}, respectively, provide another generating set of centro-affine invariants under the invariant differentiation 
 $$D_\pbal=(3\, \alpha)^{-1/3} D_s=(3\,\ah)^{-1/3 } \,D_\caal,$$ and, therefore, can be expressed in terms of $\kh$ and $\th$. 
We of course, already have such an expression for $\prcurvh$, given by  \eq{mh-proj},
and can rewrite it in the alternative form using $D_\pbal\,$:
 \beq\label{mh-proj2}\prcurvh= \frac{\ah\,\ah_{\pbal\pbal}-\fr56\, \ah_\pbal^2}{\ah^2} \ + \ \frac32\;\frac{(3\,\ah)^{1/3 }\,\kh_{\pbal} + \f4\: \kh^2 - 1}{(3\,\ah)^{2/3 }}.
\eeq
We further find that
\beq\label{mh-proj2Z}
\ihZ_1=-\frac{\,\ah_\caal+\frac 3 2 \kh\,\ah}{3^{4/3}\,\ah^{4/3}}={-\frac 1 3 \,\frac{\ah_\pbal}{\ah}-\frac 1 6\,\, \frac{3^{2/3}\,\kh}{\,\ah^{1/3}}}.\eeq

On the other hand, $\prcurvh$ and $\itZ=\invt(Z)$, given by \eq{eq-invtz}, provide an alternative generating set of centro-equi-affine  invariants under the invariant differentiation $D_\pbal$. We can express these invariants in terms of $\kappa$ and $\tau$  {(or, rather $\kappa$ and $\alpha$)}  and their derivatives with respect to $D_s$. We find that 
\beq\label{eq-itZ}\tag{3.30*}\itZ= \frac 1{\left(3 \,\alpha\right)^{1/3}}\eeq
 and comparing with \eq{eq-invtz}, we observe that the expression $3z^3\alpha$, evaluated at a point on a space curve $\C$, equals $\eacurv_\eaal$, the derivative of the equi-affine curvature with respect to equi-affine arc-length evaluated at the corresponding point of its  projection.
 The formula for  $ \prcurvh$ becomes rather simple:
 \beq\label{prca}\tag{3.31*}\prcurvh= \frac{\alpha_{ss}\,\alpha-\fr76\,\alpha_{s}^2-\fr32\,\kappa\,\alpha^2\ostrut06}{ 3^{2/3} \,\alpha^{8/3}}\eeq
and can be compared with formula \eqref{prcurv} for the projective curvature in terms of the planar  equi-affine invariants.  If we replace $\alpha$ by  $\eacurv_\eaal$  and $\kappa $ by  $\eacurv$ in the above formula, we obtain a very similar formula to \eqref{prcurv} --- the  difference is in the overall factor and also in  the coefficient of the last term in the numerator. {In part this may be explained  by the fact that $\eacurv_\eaal=3z^3\alpha$, as observed above.   
The centro-affine invariant \eqref{mh-proj2Z} has a particular simple expression in terms of centro-equi-affine invariant $\itZ$, or, equivalently,  $\alpha$:
\beq\label{mh-proj2Zb}\tag{3.32*}\ihZ_1={\itZ_s}=-\frac{\alpha_s}{\left(3 \,\alpha\right)^{4/3}}\eeq}

{We finally note that we can also write
\beq\label{etahe}\tag{3.33*} \prcurvh = -3 \itZ \,\itZ_{ss} + \fr32 \itZ_s^2 - \fr32 \itZ^2\kappa 
= -6\, \itZ^{3/2} (\itZ^{1/2})_{ss} - \fr32 \itZ^2\kappa  = -6\, \itZ^{3/2} (D_s^2 + \f4\: \kappa )\itZ^{1/2}.\eeq
Alternatively, since
\beq\label{Ds}\tag{3.34*}D_s=\frac 1 \itZ\, D_\pbal,\eeq
we can rewrite \eq{etahe} as
\beq\label{etahg}\tag{3.35*} \prcurvh = \fr32\,\itZ^{-2}\, (3 \,\itZ_\pbal^2-2\, \itZ\,\itZ_{\pbal\pbal}-\itZ^4\,\kappa)\eeq
and then solve for $\kappa$:
\beq\label{kappa}\tag{3.36*}\kappa=-\frac 1 3 \,\frac {2\, \prcurvh \,\itZ^2+6\, \itZ\,\itZ_{\pbal\pbal} -9\, \itZ_\pbal^2}{\itZ^4\ostrut{13}0} \eeq 
Formulae  \eq{eq-itZ}, \eq{prca}, \eq{Ds}, and \eq{kappa} give strikingly simple relationships between two natural generating sets of the differential algebra of centro-equi-affine invariants:\hfill\break
\vskip-10pt
\hglue1in (a) \ $\kappa$ and $\alpha$ under $D_s$;
\hskip1in (b) \ $\prcurvh$ and $\itZ$ under $D_\pbal$.
\hfill\break
\vskip-10pt
\noindent The first set is naturally expressed in terms of the position vector of a curve and its derivatives, while the second set has a natural relationship with the invariants of the image of the curve under projective and equi-affine actions on the plane. Indeed, recall that $\prcurvh$ and $d\pbal$ are the projective curvature and arc length element, respectively, of the image curve, while, from \eq{eq-invtz}, \eq{eq-c-can}, $\itZ={z\,\eacurv_\eaal^{-1/3}}$, where  $\eacurv_\eaal$ is the derivative of the equi-affine curvature with respect to the equi-affine arc length.}

\subsection{Projections centered at an arbitrary point}
\label{sect-Pc}

We now consider more general central projections of space curves.
Let $\Pc$ be the central projection centered at the point $\ch=(c_1,c_2,c_3)$, mapping $\M = \{ (x,y,z)\,|\,z \neq c_3\} \subset \R^3$
 to the plane $N = \{z=1+c_3\} \simeq \R^2$. Explicitly,
\Eq{Pc}
$$(X,Y)=\Pc(x,y,z)=\left(\frac{x-c_1}{z-c_3}+c_1, \ \frac{y-c_2}{z-c_3}+c_2\right),$$
where $\Po$ given by \eq{P0} is the special case when $c_1=c_2=c_3=0$.

We denote the space translation by the vector $\ch$ as $T_{\ch}\colon \M \to\M $ and the plane translation by the vector $\c=(c_1,c_2)$ as  $\Tc\colon \N \to\N $. Clearly
\beq\label{eq-Pc}\Pc=\Tc\,\Po \,T_\ch^{-1}.\eeq 
Although the map \eq{eq-Pc} involves an extra transformation $\Tc$  that does not appear in the map $\P_g$ defined in Theorem~\ref{th-family}, an almost identical  proof implies that the action  of $\GL3_\ch = \Ad T_\ch (\GL3)$ on $M$ is $\Pc$-projectable. Explicitly:


 \begin{proposition}  \label{prop-stab2} For any  non-singular linear transformation $A\in \GL3$ acting on $\M $,
 $$ (T^{\vphantom1}_{\c}\,[A]\,\Tc^{-1})\, \Pc =\Pc\,(T^{\vphantom1}_\ch \,A\,T_\ch^{-1}),$$
where $[A]\in \PGL3$ is the corresponding projective transformation on $\N$.
\end{proposition}

\begin{proof} The $GL(3)$-action on $M$ is $\Po$-projectable and, moreover, satisfies
\beq\label{lem-stab}
 [A] \,\Po =\Po\, A\ \ \qbox{for all} \ \ A \in \GL3.
  \eeq
We substitute   $\Po=\Tc^{-1}\,\Pc \,T_\ch$, obtained from \eq{eq-Pc}, to complete the proof.
\end{proof} 

As before, to determine the formulas for the induced projection $\Pck$ on curve jets, 
we choose a  representative smooth curve  $\C \subset \M$, parametrized by $(x,y(x),z(x))$, such that $\rr^{(k)} = \jk \C \at{\rr_0}$ at the point $\rr_0 = (x_0, y(x_0),z(x_0))$.  Then its central projection  $\Cp =\Pc(\C)\subset\N $ has parametrization
 \beq \label{eq-pc}\left(X(x),Y(x)\right)=\left(\frac {x-c_1}{ z(x)-c_3}+c_1, \ \frac {y(x)-c_2}{ z(x)-c_3}+c_2\right).\eeq
At the image point $\w_0 = \Pc(\rr_0) $, the projected curve jet is $\w^{(k)} = \Pk _\ch(\rr^{(k)}  ) = \jk \Cp\at{\Pc(\rr_0)}$.   Proposition~\ref{prop-stab2} and Theorem \ref{th-dinv} imply:

\begin{theorem} \label{th-inv-c} If $I\colon \Jk(\N,1)\to\R$ is a differential invariant for the projective action of $\PGL3$ on $\N $, then $\Ih = I\comp \Pck\colon \Jk(\M,1)\to\R$ is a differential invariant for the translational conjugation  
$\GL3_\ch := \Ad T_\ch (\GL3)$ 
of the centro-affine action.
\end{theorem}
 
 \Rmk As before, Theorem~\ref{th-inv-c} remains valid if we replace $\Pc$ with \emph{a projection centered at $\ch$ to an arbitrary plane}, because the projected images of a space curve with the \emph{same} center are all related by projective transformations.

The pull-back $\prcurvh_{\,\ch} = \Pc\pss7\,\prcurv$ of the planar projective curvature \eq{prcurv} is a $\GL3_\ch\,$--invariant. According to \eq{eq-PoPg},  $\prcurvh_{\,\ch}=\prcurvh\comp T_\ch^{-1}$, where  $\prcurvh =\Po\pss7\,\prcurv$ can be expressed in terms of the normalized invariants for the centro-affine action on $\R^3$.
In particular, \eqf{mh-proj1} expresses $\prcurvh$ in terms of the normalized invariants $I_k,J_k$ corresponding to the cross-section $\K$ given by  \eq{eq-cs}.  Then, according to Proposition~\ref{prop-conj-inv}, the functions  $I_{\ch,k}=I_k \comp T_\ch^{-1}$ and $J_{\ch,k}=J_k \comp T_\ch^{-1} $ are $\GL3_\ch\,$-invariants obtained by  invariantization of  $y_k \comp T_\ch^{-1}$ and $z_k \comp T_\ch^{-1} $  relative to the cross-section: 

\beq\label{eq-csc}
\eeqo{\K_\ch=T_\ch\, (\K) \\= \{\qaeq{5}{x=c_1,\\y=c_2,\\z=1+c_3,\\y_1=0,\\z_1 = 0,\\y_2 = 1,\\z_2 = 0,\\y_3 = 0,\\z_3 = 1}\}.}
\eeq
{Taking into account that translations leave jet variables of the first order and higher invariant (i.e.~$y_k \comp T_\ch^{-1}=y_k$ and $z_k \comp T_\ch^{-1} =z_k$, for $k\geq 1$) we observe that $I_{\ch,k}$  and $J_{\ch,k}$ are, in fact, normalized invariants.}
The projective curvature $\prcurvh_{\,\ch}$ of the projected curve \eq{eq-pc} can then be re-expressed in terms of these invariants by simply replacing, in \eq{mh-proj1}, each $I_k,J_k$ with the corresponding invariant $I_{\ch,k},\,J_{\ch,k}$.

\subsection{The standard parallel projection}
\label{sect-P0-p}


 By the \emph{standard   parallel projection}, we mean the orthogonal projection from $M = \R^3 $  to  the $xy$-plane $\N=\R^2 $. We use the coordinates $(X,Y)$ on the image plane that agree with the corresponding rectangular coordinates on $\R^3$, i.e., $X(x,y,0)=x$ and  $Y(x,y,0)=y$. The resulting parallel projection map $\Po\colon\M \to\N$ is explicitly given by
 \beq\label{P}
 (X,Y)=\Po(x,y,z)=\left(x, y\right).
 \eeq
 It is easily checked that the maximal $\Po$-projectable subgroup $H \subset G =\AG3$  consists of the transformations
 \beq\label{eq-pp0}
(x,y,z) \ \longmapsto \ (a_{11}\, x+a_{12}\,y+c_1,\  a_{21}\, x+a_{22}\,y+c_2,\ a_{31}\, x+a_{32}\,y+a_{33}\,z+c_3)
\eeq
 where $a_{33}(a_{11}\,a_{22}-a_{12}a_{21})\neq 0$.
 This action projects to the affine action
\beq \label{eq-a2} (X,Y)\ \longmapsto\ (a_{11}\, X+a_{12}\,Y+c_1,\ a_{21}\, X+a_{22}\,Y+c_2)
\eeq
on  $\N=\R^2$.
 The global isotropy group $H_N$ consists of the transformations
 \beq\label{eq-pp0HN}
(x,y,z) \ \longmapsto \ (x,\ y, \ a_{31}\,x+a_{32}\,y+a_{33}\,z+c_3)
\eeq
that fix the points on the $xy$ plane.

We now investigate the prolonged action on the curve jet spaces and the consequential differential invariants.  
If  $\C \subset \M$ is a smooth space curve parametrized by $(x,y(x),z(x))$, then its projection  $\Cp =\Po(\C)\subset\N $ has parametrization
 \beq\label{eq-pXY} \left(X(x),Y(x)\right)=\left(x, y(x)\right).\eeq
Applying the transversality condition \eq{DPmax} to \eq{P},  we see that all jets are $\Po$-regular,  and thus the prolongation $\Pok\colon \Jk(\M ,1)\to\Jk(\N ,1) $ is globally defined by
 $$
 \qxeq{X=x,\\ Y=y,\\ Y_1=\frac{D_x Y}{D_x X}=y_x ,\\  Y_i=\frac{D_x Y_{i-1}}{D_x X} =y_i,\\ i>1. }$$ 
This induces an obvious isomorphism between the algebra of fiber-wise constant differential invariants of the action \eq{eq-pp0} and the algebra of affine differential invariants of planar curves. 

We can use paradigm of \sect{sect-Pi-iota} to construct a cross-section on $\M$ that projects to the standard cross-section for the affine planar action:
\Eq{scsapa}
$$\K=\{\xeq{X=0,\\ Y=0, \\ Y_1=0, \\Y_{2}=1, \\Y_{3}=0,\\ Y_4=3}\}.$$ 
The moving frame invariantization associated to this cross-section produces the affine curvature invariant 
 \beq\label{eq-affcurv}\affcurv =\iota(Y_5)= 3\,\frac{\A}{\B^{3/2}}\,,\eeq
where $\A,\B$ are given by formulas \eq A, \eq{eacal}, respectively,
along with the contact-invariant arc-length element and its dual invariant differential operator 
\beq\label{eq-a2-ds}
\qqeq{{d\affal}=\piH\iota(dX)={\frac 1 3} \,\frac{\B^{1/2}}{Y_2}\,dX,\\\CD_\affal={3}\,\frac{Y_2}{\B^{1/2}}\,D_X={\inv(D_X)}.}
\eeq

The recurrence formulae 
then express the higher order normalized invariants in terms of invariant derivatives of the affine curvature \eq{eq-affcurv}, namely:
\Eq{affcurvderivs}
$$\qqeq{\iota(Y_6)= \nu_\sigma+ \f2\: \nu ^2 + 45, \\\iota(Y_7)= \nu_{\sigma\sigma} + \fr53\: \nu \, \nu_\sigma + \f3\: \nu ^3 + 51\: \nu , }$$
and so on.
These are all defined on the $\AG2$-invariant open subset of the jet bundle prescribed by the inequality $\B>0$. 

We note that the group $H$ acting on $\R^3$ by \eq{eq-pp0}  is a product of two groups, namely, $\widetilde H=\AG2$ acting by 
\beq
 \label{eq-a2-lift} 
(x,y,z)\ \longmapsto \ (a_{11}\, x+a_{12}\,y+c_1,\ a_{21}\, x+a_{22}\,y+c_2,\ z)
 \eeq
and  $H_N$ acting by
\beq\label{eq-GN}
(x,y,z)\ \longmapsto \ (x,\ y,\ a_{31}\, x+a_{32}\,y+a_{33}\,z+c_3).
\eeq 
The invariants of the $\AG2$-action \eq{eq-a2-lift}  can be obtained by lifting   the cross-section \eq{scsapa} to $\R^3$, producing 
\beq\label{tK}\Kt=\{\xeq{x=0,\\ y=0, \\ y_1=0, \\y_{2}=1,\\ y_{3}=0,\\ y_4=3}\}.\eeq
The corresponding normalized differential invariants $\invt (y_i)$, $i\geq 5$, are obtained by replacing the capital letters $Y$ and $X$ with their lower case versions $y$ and $x$, respectively, in \eq{eq-affcurv}, \eq{affcurvderivs}. The invariant differential form $d \widehat \affal =\pi_H\invt(dx)$ and dual invariant differential operator 
\beq
\label{eq-inv-diff-aff}\CD_{\widehat \affal} =\invt(D_x)=\frac{3\,y_2}{\sqrt{3\,y_2\,y_4-5\,y_3^2}}\,D_x\eeq  are also obtained  in the same manner from \eq{eq-a2-ds}.

We can also employ the recurrence formulae to determine the higher order differential invariants
\beq
\label{eq-a2-inv-1}
\eeqo{\invt (z) = z,\qquad \qquad \invt(z_1)=z_{\widehat \affal} = \frac{3\, y_2\,z_1}{\sqrt{3\, y_2\,y_4-5\,y_3^2}},\\ \invt (z_2)= z_{\widehat \affal \widehat \affal} + \f6\: \nu \, z_{\widehat \affal} = \frac{3\,y_2\,(3\,y_2\,z_2-y_3\,z_1)}{3 \,y_2\,y_4-5\,y_3^2}\,,\\
\invt(z_3) = z_{\widehat \affal \widehat \affal \widehat \affal} + \f2\: \nu \, z_{\widehat \affal \widehat \affal} + \bigl(\f6\: \nu_{\widehat \affal} + \f{18}\, \nu ^2 + 1\bigr)z_{\widehat \affal}   = \frac{27 \,y_2^2\,(y_2\,z_3-y_3\,z_2)}{(3 \,y_2\,y_4-5\,y_3^2)^{3/2}}\,,
\\
\invt(z_4)= z_{\widehat \affal \widehat \affal \widehat \affal \widehat \affal} + \nu \, z_{\widehat \affal \widehat \affal \widehat \affal} + \bigl(\fr23\: \nu_{\widehat \affal} + \fr{11}{36}\, \nu ^2 + 4\bigr) z_{\widehat \affal \widehat \affal} + \bigl(\f6\: \nu_{\widehat \affal\widehat \affal} + \fr7{36}\: \nu \: \nu_{\widehat \affal} + \f{36}\, \nu ^3 + \nu\bigr)z_{\widehat \affal}
\\
= \frac{81\,y_2^3\,(y_2\,z_4-2\,y_3\,z_3)+27\,y_2^2\,y_3^2\,z_2-(27\,y_2\,y_4-45\,y_3^2)\,y_2\,y_3\,z_1
}{(3 \,y_2\,y_4-5\,y_3^2)^{2}}\,,}
\eeq 
and so on.
The  algebra of differential invariants for the action \eq{eq-a2-lift}  is generated by the gauge invariant 
\beq\label{eq-affcurvh} \affcurvh = \invh(y_5) = 3\,\,\frac{9\,y_5\, y_2^2-45\,y_4\,y_3\,y_2+40\,y_3^3}{(3\,y_2\,y_4-5\, y_3^2)^{3/2}}
\eeq
 and the non-gauge invariant  $z$ through invariant differentiation under $\CD_{\widehat \affal}$.  

The prolongation of the  $H_N$-action  \eq{eq-GN} is given by
\beq\label{eq-prGN} 
\qeq{y_1\longmapsto y_1,\\ 
z_1 \longmapsto  a_{31}+a_{32}\,y_1+a_{33}\,z_1,\\
 y_i\longmapsto y_i,\\
 z_i \longmapsto  a_{32}\,y_i+a_{33}\,z_i, \\ i>1.}
\eeq
The action restricts to the lifted cross-section \eq{tK} as follows:
 \begin{eqnarray}\label{eq-prGN-K} 
\seq{\qeq{y_i\longmapstox y_i,\\ z\longmapstox a_{33}\,z+c_3,\\
 z_1\longmapstox a_{31}+a_{33}\,z_1,\\
 z_2\longmapstox a_{32}+a_{33}\,z_2,}\\
\qeq{
  z_3\longmapstox a_{33}\,z_3,\\
 z_4\longmapstox 3\,a_{32}+a_{33}\,z_4,\\
 z_i\longmapstox a_{32}y_i+a_{33}\,z_i,\\i>4.}}
\end{eqnarray} 
 
  We can follow the inductive approach of 
  \cite{kogan-03} 
  to express the invariants of the $H$-action \eq{eq-pp0} in terms of the invariants of the $\AG2$-action \eq{eq-a2-lift}.
We choose the cross-section $\widehat \K \subset  \widetilde\K$ to the action \eq{eq-prGN-K} of $H_N$ defined by
 $$\qaeq{13}{x=0,\\ y=0, \\ y_1=0,\\ y_{2}=1,\\ y_{3}=0,\\ y_4=3,\\ z=0,\\ z_1=0,\\z_2=0,\\z_3=1},$$ 
 which can be proven to also be a cross-section of the $H$-action on $\R^3$.   
The induced moving frame normalizations are
$$\qqeq{ a_{33}=\frac 1 {z_3},\\ a_{32}=-\frac {z_2} {z_3},\\ a_{31}=-\frac {z_1} {z_3},\\
    c_{3}=-\frac {z} {z_3}.}$$
 Using formulas (22) and (25) in   \cite{kogan-03}, we obtain the following  normalized invariants for the $H$-action
\beq
\qeq{\invh(y_i) = \invt (y_i),\\
\invh(z_4) = \frac{\invt(z_4)-3\,\invt (z_2)}{\invt (z_3)}\,,\\
\invh(z_i) = \frac{\invt(z_i)-\invt(z_2)\,\invt (y_i)}{\invt (z_3)}\,,\\ i>4}.
\eeq  
As expected $\invh(D_x)=\invt(D_x)=\CD_{\widehat \affal}$,  given by \eq{eq-inv-diff-aff}.
 The differential invariants of the $H$-action \eq{eq-pp0} are  generated by the gauge invariant $\affcurvh$, given by \eq{eq-affcurvh},
and the non-gauge invariant
$$
\invh(z_4) = 3\,\,\frac{y_2\,(y_2\,z_4-y_4\,z_2)-2\,y_3\,(y_2\,z_3-y_3\,z_2)}{(y_2\,z_3-y_3\,z_2)\,\sqrt{3\,y_2\,y_4-5\,y_3^2}},
$$
 through invariant differentiation under $\CD_{\widehat \affal}$.
%
\subsection{Family of parallel projections}\label{sect-Pc-p}
Using general framework of Section~\ref{families}, we now consider the family of parallel projections from $M =\R^3$ to the $x\,y$-plane $N =\R^2$ in the direction of the vectors $\b=(b_1,b_2,1)$.  We assume that the coordinate functions $(X,Y)$ on the image plane agree with the corresponding coordinates on $\R^3$, i.e.~$X(x,y,0)=x$ and  $Y(x,y,0)=y$. The resulting projection $$\Pb\colon\M=\R^3 \ \longrightarrow\ \N=\R^2 $$ is explicitly given by
 \beq\label{Pb}
 (X,Y)=\P_\b(x,y,z)=\left(x-b_1z,\ y-b_2z\right).
 \eeq
 Let $T_\b\in A(3)$ denote the linear transformation on $\R^3$ given  by  \beq\label{eq-b} 
 T_\b\colon  (x,y,z) \ \longmapsto \ (x + b_1z,\ y+b_2z,\ z).\eeq
  Obviously $\Pb= \Po\comp T_\b^{-1}$, and the  action of $H_\b=T_\b\,H\,T_\b^{-1}\subset A(3)$  is $\Pb$-projectable. 
  According to \eq{eq-PoPg}, the pull-back of the planar  affine curvature is given by
  \beq\label{aff-curv-b}\affcurvh_\b= \Pb\pss5\,\affcurv =\affcurvh\comp (T_\b^{-1})^{(5)}, \eeq
  where $\affcurv$ is given by \eq{eq-affcurv} and $\affcurvh$ is given by \eq{eq-affcurvh}. The resulting expression is rather complicated, involving $y_i, z_i$ for $1 \leq i \leq 5$, and $b_1,b_2$, and is obtained by replacing the $y_i$ in $\affcurvh$ with their pull-backs under the prolonged $T_\b^{-1}$-action. For instance, $y_2$ must be replaced with 
  $$y_2\comp (T_\b^{-1})^{(2)}=\frac{y_2\,(1-b_1z_1)+z_2\,(b_1y_1-b_2)}{(1-b_1z_1)^3}.$$
 On the other hand,  in accordance with \eqref{tildeiota} in Proposition~\ref{prop-conj-inv},  $\affcurvh_{\:\b}=\invh_{\!\b}(y_5 \comp {T_b^{-1}} )$, where $\invh_{\!\b}$ is the $H_\b$-invariantization corresponding to the cross-section $\Kh_\b = T_\b(\Kh)$ defined by 
 $$\qaeq{6}{x = 0,& y = 0,& y_1 = 0,& y_2 = 1,& y_3 = b_2,& y_4 = b_2z_4-4\:b_1+3, & z = 0, &z_1 = 0,& z_2 = 0,& z_3= 1.}$$
Combining \eqref{tildeiota} with the Replacement Theorem, we can compute explicit relations between normalized invariants for invariantizations   $\invh_{\!\b}$   and $\invh$. For example,\footnote{Formulae \eq{ind1} and \eq{ind2} do not appear in the version of the paper published in  {\it Lobachevskii J.~Math.} {\bf 36} (2015), 260--285. }  
\beq\label{ind1}\invh(y_5)  =\invh_{\!\b}( y_5 +5\,b_1\,z_4-b_2\, z_5) \comp {T_\b}^{(5)},\eeq
while
\beq\label{ind2}\invh_{\!\b}(y_5)=\invh(y_5-5\,z_4\,b_1+z_5\,b_2-10\,b_1\,b_2)\comp(T_\b^{-1})^{(5)}.\eeq
 
 Although explicit general formulae  for the invariants $ \affcurvh_{\b}$ become cumbersome, \eq{aff-curv-b} provides a useful relation between the invariants of a space curve $\C$ and the affine curvature of the images $\Cp_{\b}$    under parallel projections in various directions, as specified by the vector $\b=(b_1,b_2,1)$. These quantities are easily computable for \emph{a specific curve} $\C$ and could be of use in applications to the problem of reconstruction of an object from its various images.      


\Section4{Conclusion and future work}

In this paper, we examined the relationships between the differential invariants  of objects and of their images  under a surjective map  $\P\colon M\to N$. Our analysis covers both the case when the underlying transformation group $G$ maps fibers of $\P$ to fibers, and therefore projects to a group action on $N$,  and the case when only a proper subgroup $H\subset G$ acts projectably. In the projectable case, we  established an explicit, constructible isomorphism between the algebra of differential invariants on $N$ and the algebra of fiber-wise constant (gauge) differential invariants on $M$. This isomorphism leads to explicit formulae for the invariants of the image of a submanifold  $S\subset M$ in terms of invariants of $S$. In particular, we expressed the projective curvature of a planar curve in terms of centro-affine invariants of its pre-image under the standard central  projection from $\R^3$ to $\R^2$.   In the non-projectable case, we introduced a family of surjective maps $\P_g$, parametrized by elements of $g\in G$, and then expressed the differential invariants of each $\P_g$-image of a submanifold of $S \subset M$ in terms of its $\Ad g\, (H)$-invariants which,  in turn, can be easily obtained from its  $H$-invariants.

Motivations for considering both projectable and non-projectable actions comes from basic problems arising in image processing: establishing relationships
between three-dimensional objects and their two-dimensional images and  reconstructing an object from  its various  images.  
In \cite{BurdisKogan, BuKoHo}, differential signatures of families of  planar curves were used to obtain a novel algorithm for deciding whether a given planar curve is an image of a given space curve, obtained by a central or a parallel projection with unknown parameters. 
In this paper, we  establish the relationship between differential invariants of a space curve  and its various projections.
In this context, further analysis of the effect of a surjective map on the associated  differential invariant signatures, used in object recognition and symmetry detection, \cite{COSTH}, is worth pursuing.
These results may also find applications in the problems of high dimensional data analysis, by studying projections of the data to lower dimensional  subspaces. 
Applying the methods developed in the paper to these problems is one of the directions of future research.

\vskip.3in

{\bf Acknowledgement:} Irina Kogan gratefully acknowledges the support and hospitality provided by the Institute for Mathematics and its Applications (IMA) at the University of Minnesota, where this work was initiated during her sabbatical stay in the 2013--14 academic year. 


\end{document}

\end{document}